\documentclass[reqno,11pt]{amsart}

\usepackage{times,amsmath,cancel,stmaryrd,graphicx}
\usepackage{amsfonts,enumitem,amssymb,color,mathrsfs}
\usepackage[colorlinks=true]{hyperref}
\usepackage{cleveref}
\usepackage{lmodern}
\usepackage[dvipsnames]{xcolor}
\usepackage{comment}
\usepackage{tikz}
\usetikzlibrary{arrows.meta,calc}
\usepackage{cite}
\usepackage{float}

\usepackage{mathtools}   
\usepackage{dsfont}      
\usepackage{booktabs}    
\graphicspath{{Images/}{../Images/}}

\newtheorem{thm}{Theorem}[section]
\newtheorem{lem}[thm]{Lemma}

\newtheorem{prop}[thm]{Proposition}

\newtheorem{rem}[thm]{Remark}

\numberwithin{equation}{section}
\newcommand{\R}{\mathbb{R}}

\newcommand{\N}{\mathbb{N}}
\newcommand{\rd}{\mathrm{d}}
\newcommand{\sign}{\mathrm{sign}}

\allowdisplaybreaks

\begin{document}

\title[Coral--Macroalgae Dynamics under Crowding]{Dynamics of Coral--Macroalgae Interactions under Crowding}
\author{Julie C. Blackwood}
\address{Department of Mathematics\\
Williams College\\
180 Main Street\\
Williamstown, MA 01267\\
USA}
\email{jcb5@williams.edu}

\author{Katerina Nik}
\address{King Abdullah University of Science and Technology (KAUST)\\
CEMSE Division\\
Thuwal 23955-6900\\
Saudi Arabia}
\email{katerina.nik@kaust.edu.sa}

\author{Simon Nik}
\address{King Abdullah University of Science and Technology (KAUST)\\
CEMSE Division\\
Thuwal 23955-6900\\
Saudi Arabia}
\email{simon.nik@kaust.edu.sa}

%

\begin{abstract}
We study a planar ODE model for the benthic competition between coral, macroalgae, and algal turf on a reef, extending the classical model of Mumby, Hastings, and Edwards by a nonlinear, density-dependent coral mortality that accounts for crowding. The strength of crowding is set by an exponent $\delta>0$ that reshapes the coral nullcline and enriches the bifurcation structure of the system. We establish positive invariance of the biologically relevant region and the absence of periodic orbits, classify the three boundary equilibria together with their local stability, and reduce the coexistence problem to a single scalar equation whose shape, in particular its concavity, controls the number and local stability of the interior equilibria. The grazing intensity $g$ organizes the dynamics through two thresholds $g_0<g_1$ determining the stability of the coral- and macroalgae-dominated states, and a further threshold $g^\star$ at which two interior equilibria collide. We prove that the system undergoes transcritical bifurcations at the boundary equilibria and a saddle-node bifurcation of interior equilibria, and we discuss the implications for coral reef resilience and hysteresis. We complement these results with numerical simulations that illustrate the bifurcation sequence across the grazing regimes.
\end{abstract}
%
 %

\maketitle

\section{Introduction}\label{sec:intro}

Coral reefs are among the most productive and diverse marine ecosystems, providing essential services that range from coastal protection and fisheries to tourism~\cite{MoFo99}. Across the globe, however, these ecosystems are increasingly threatened by stressors that are often exacerbated by anthropogenic activity: rising sea surface temperatures induce coral bleaching, while the over-harvesting of herbivorous reef fish promotes algal overgrowth that outcompetes and ultimately damages the coral~\cite{Ban13,Harborne17,Hughes99}.

There is both empirical and theoretical evidence that coral reef ecosystems admit two alternative stable states: a desirable state of high coral cover and a degraded, coral-depleted state dominated by macroalgae~\cite{Bell04,Hugh94,Know92,Mum07}. A well-documented illustration occurred in the Caribbean in the 1980s, when a disease-induced mass mortality of the sea urchin {\it Diadema antillarum}, until then the principal grazer, was followed by a marked decline of coral and a proliferation of algae~\cite{Lessios84}. In the absence of urchins, herbivorous reef fish such as parrotfish became the primary grazers, their scraping of algae from hard substrates helping to curb macroalgal overgrowth.

Maintaining the desirable state of high coral cover requires managing several interacting stressors simultaneously. Phase shifts between the two states may occur abruptly as exogenous conditions, such as the grazing rate, vary, and reversing them may then require improving these conditions well beyond the threshold at which the collapse occurred. This hysteresis underscores the importance of sustaining coral resilience~\cite{Black12,Fun11}. Alongside the management of grazing pressure, considerable effort is therefore devoted to active restoration, such as coral transplantation, coral gardening, and the deployment of artificial reefs~\cite{Bayraktarov19}. Mathematical models of coral--macroalgae competition, in turn, have become a central tool for understanding the transitions between these states and the conditions under which a degraded reef may recover~\cite{Black12,Mum07,ZhYu22}.

In the present work we adapt the well-established Mumby--Hastings--Edwards model~\cite{Mum07}, which captures the central feedbacks between coral and algae competing for space on the reef, by incorporating the nonlinear effects of {\it coral crowding}: as coral cover increases, intraspecific competition raises the coral mortality rate. We model this density dependence through a mortality of the form $m(C)=m_0+dC^{\delta}$, where $m_0>0$ is the baseline mortality and the exponent $\delta>0$ tunes how steeply mortality grows with coral cover. Although such density dependence is ecologically well documented~\cite{Hughes84}, its consequences for the qualitative dynamics of coral--macroalgae models have received comparatively little rigorous attention. When $d=0$ the model reduces to the constant-mortality system of~\cite{Mum07,Li14}, whose coral nullcline is a straight line. The crowding term reshapes this nullcline for every $\delta>0$. As we show below, intermediate crowding $1<\delta<2$ renders the reduced equation non-concave and can produce three coexistence equilibria and a tristable reef, a feature the constant-mortality model cannot exhibit~\cite{Black12,Li14,Mum07,ZhYu22}. Ecologically, such a tristable reef can settle into coral dominance, macroalgae dominance, or a mixed coral--macroalgae community, with the initial cover deciding the outcome.

The aim of the present work is to provide a rigorous qualitative analysis of the resulting planar system, with particular emphasis on how the crowding exponent $\delta$ reshapes the equilibrium and bifurcation structure. We determine all boundary and interior equilibria, characterize their local stability, identify the grazing thresholds at which the reef transitions from macroalgae dominance through bistability to coral dominance, and complement this analysis with numerical simulations. 
In particular, we reduce the coexistence problem to a single scalar equation whose concavity controls the number and local stability of the interior equilibria. The non-concave regime $1<\delta<2$, in which this reduction admits as many as three coexistence states, requires a more delicate analysis. We refer to~\cite{Fun11,Mum07,TanLanWei24} for the ecological background and for the parameter ranges underlying the model.\\

The article is organized as follows. In Section~\ref{sec:model} we formulate the model, reduce it to a planar system on a positively invariant region~$\Omega$, and show the absence of periodic orbits (Lemma~\ref{lem:dulac}). In Section~\ref{sec:boundary} we determine the three boundary equilibria and characterize their local stability in terms of two grazing thresholds $g_0<g_1$ (Lemma~\ref{lem:stabO} and Theorems~\ref{thm:stabB} and~\ref{thm:stabC}). In Section~\ref{sec:interior} we reduce the coexistence problem to a scalar equation $L(y)=0$, relate the local stability of an interior equilibrium to the sign of $L'(y)$, and give a complete classification of the dynamics across the low-, intermediate-, and high-grazing regimes, the last controlled by a further threshold~$g^\star$ (Theorem~\ref{thm:classification}). In Section~\ref{sec:bif} we rigorously establish the transcritical bifurcations at the boundary equilibria $B$ and $C$ (Theorems~\ref{thm:transcritB} and~\ref{thm:transcritC}) and the saddle-node bifurcation of interior equilibria (Theorem~\ref{thm:saddlenode}), deferring to Appendix~\ref{app:proofs} the accompanying normal-form reductions, which identify the local type of each degenerate equilibrium at its threshold as a saddle-node point. In Section~\ref{sec:exnum} we specialize the analysis to the linear, quadratic, and intermediate crowding exponents and illustrate the resulting dynamics, including the full bifurcation sequence, numerically. Section~\ref{sec:conclusion} discusses the ecological implications for coral reef resilience.


\bigskip

\section{The Coral Reef Model with Crowding}\label{sec:model}

We consider a benthic reef community composed of three constituents that compete for space: coral~$C$, macroalgae~$M$, and algal turf~$T$, with their interactions depicted in Figure~\ref{fig:schematic}. Each variable denotes the fraction of the seabed occupied by the respective constituent, and the seabed is assumed to be fully covered, so that
\begin{equation}\label{cover}
C+M+T=1\,.
\end{equation}
Building on the classical model of~\cite{Mum07}, in which coral and algae compete through overgrowth, grazing, and natural turnover, we incorporate a density-dependent coral mortality that accounts for crowding. The resulting dynamics are described by the system of three coupled nonlinear ordinary differential equations (ODEs)
\begin{subequations}\label{CMT}
\begin{align}
\dot M &= a MC + \gamma MT - \frac{g M}{M+T} - n M\,, \label{CMT1}\\
\dot C &= bCT - aMC - \big(m_0 + d C^{\delta}\big)C\,, \label{CMT2}\\
\dot T &= \frac{g M}{M+T} + n M + \big(m_0 + dC^{\delta}\big)C - b CT - \gamma MT\,, \label{CMT3}
\end{align}
\end{subequations}
where the dot denotes differentiation with respect to time~$t$.

\begin{figure}[H]
\centering
\begin{tikzpicture}[
  >=Latex,
  comp/.style={rectangle, rounded corners=6pt, draw=black, thick,
               minimum width=2.5cm, minimum height=1.15cm, align=center, font=\small},
  flow/.style={->, thick},
  rate/.style={font=\small, fill=white, inner sep=1.5pt}
]
\node[comp, fill=blue!20]   (C) at (0,2.7)     {coral\\[1pt]$C$};
\node[comp, fill=orange!20] (M) at (-3.7,-1.5) {macroalgae\\[1pt]$M$};
\node[comp, fill=green!18] (T) at (3.7,-1.5)  {turf\\[1pt]$T$};
\draw[flow] (C) to[bend right=8] node[rate,above left]{$a$} (M);
\draw[flow] (T) to[bend left=18] node[rate,auto,swap]{$b$} (C);
\draw[flow] (C) to[bend left=18] node[rate,auto]{$m(C)$} (T);
\draw[flow] (M) to[bend left=12] node[rate,above]{$g,\,n$} (T);
\draw[flow] (T) to[bend left=12] node[rate,below]{$\gamma$} (M);
\end{tikzpicture}
\caption{Schematic of the benthic reef model~\eqref{CMT}. Coral~$C$, macroalgae~$M$, and algal turf~$T$ compete for the available space, with $C+M+T=1$. Macroalgae overgrow coral at rate~$a$, while coral overgrows turf at rate~$b$ and dies, including through crowding, at the rate $m(C)=m_0+dC^{\delta}$. Macroalgae become turf through grazing at rate~$g$ and natural death at rate~$n$, and they spread vegetatively over turf at rate~$\gamma$.}
\label{fig:schematic}
\end{figure}
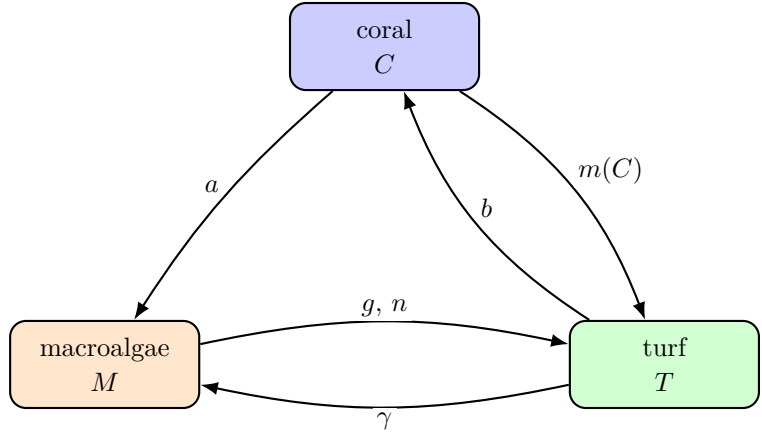

The terms have the following ecological meaning, the first five being those of~\cite{Mum07}. Macroalgae overgrow coral at rate~$a$, while grazers feed indiscriminately on macroalgae and turf at rate~$g$, so that grazed macroalgae are converted into turf. Macroalgae spread vegetatively over turf at rate~$\gamma$ and die naturally at rate~$n$, the latter producing turf, whereas coral overgrows turf at rate~$b$. The novel ingredient is the coral mortality
\begin{equation}\label{mortality}
m(C)=m_0+dC^{\delta}\,,
\end{equation}
in which corals die at a baseline rate $m_0>0$, while the term $d C^{\delta}$ models the strengthening of intraspecific (coral--coral) competition as coral cover increases, the exponent $\delta>0$ controlling how steeply the mortality grows with cover. As with the other loss terms, dead coral reverts to turf.

To remain consistent with the calibrated ranges of~\cite{Fun11}, the mortality~\eqref{mortality} is required to satisfy
\[
0.02\le m(C)\le 0.30\,,\quad C\in[0,1]\,.
\]
Since $C\mapsto m(C)$ is increasing, with $\min_{[0,1]}m=m(0)=m_0$ and $\max_{[0,1]}m=m(1)=m_0+d$, this holds if and only if
\[
m_0\ge 0.02\quad\text{and}\quad m_0+d\le 0.30\,.
\]
Representative parameter ranges and their sources are collected in Table~\ref{tab:params}.

\begin{table}[H]
\centering
\small
\setlength{\tabcolsep}{5pt}
\renewcommand{\arraystretch}{1.2}
\begin{tabular}{clccl}
\toprule
\textbf{Parameter} & \textbf{Description} & \textbf{Range} & \textbf{Unit} & \textbf{Source}\\
\midrule
$a$      & rate macroalgae overgrow coral           & $0$--$1.44$        & yr$^{-1}$ & \cite{Fun11}\\
$g$      & maximum grazing rate on algae            & $0.05$--$15$       & yr$^{-1}$ & \cite{Fun11}\\
$\gamma$ & rate macroalgae spread over turf         & $0$--$0.9$         & yr$^{-1}$ & \cite{Fun11}\\
$n$      & natural mortality rate of macroalgae     & $0$--$0.1$         & yr$^{-1}$ & \cite{ZhYu22}\\
$b$      & rate coral overgrows turf                & $0$--$1$           & yr$^{-1}$ & \cite{Fun11}\\
$m_0$    & density-independent coral mortality      & $0.02$--$0.30$     & yr$^{-1}$ & \cite{Fun11}\\
$d$      & crowding (density-dependent) coefficient & $0$--$(0.30-m_0)$  & yr$^{-1}$ & this work\\
$\delta$ & crowding exponent                        & $>0$               & ---       & this work\\
\bottomrule
\end{tabular}
\caption{Parameters of the model~\eqref{CMT}, with biologically plausible ranges and sources.}
\label{tab:params}
\end{table}

Throughout, we impose the following general assumptions:
\begin{equation}\label{GA}\tag{GA}
\begin{aligned}
& a,\,b,\,\gamma,\,g,\,n,\,m_0,\,d,\,\delta\in\R_{>0}\,,\quad b>m_0\,,\quad a<\gamma\,, \\[1ex]
& 0<n<\min\Big\{a, \gamma-g, \tfrac{\gamma(a+m_0)}{a+b}\Big\}\,, \\[1ex]
& 0<m_0+d<\min\Big\{\tfrac{b}{2(a-n)}\Big(\tfrac{\gamma(a+m_0)}{a+b}-n\Big), \tfrac{b}{\sqrt{2(\gamma-a)}}\sqrt{\tfrac{\gamma(a+m_0)}{a+b}-n}\Big\}\,.
\end{aligned}
\end{equation}
Unless stated otherwise, all results below assume~\eqref{GA}, whose last condition is used in Section~\ref{sec:boundary} to establish the strict ordering of the two grazing thresholds $g_0$ and $g_1$.

\smallskip

\subsection{Reduction to a planar system}\label{subsec:reduction}

Since $T=1-C-M$ by~\eqref{cover}, equation~\eqref{CMT3} is redundant, and using $M+T=1-C$ the dynamics reduce to the planar system
\begin{subequations}\label{dCdM}
\begin{align}
\dot M &= a MC + \gamma M(1-C-M) - \frac{g M}{1-C} - n M\,, \\
\dot C &= b C(1-C-M) - a MC - \big(m_0 + d C^{\delta}\big)C\,.
\end{align}
\end{subequations}
Writing $x:=M$ and $y:=C$ for the macroalgae and coral cover, \eqref{dCdM} takes the form
\begin{subequations}\label{dxdy}
\begin{align}
\dot x &= x\Big[\gamma(1-x) - n + (a-\gamma)y - \frac{g}{1-y}\Big] =: xF(x,y)\,, \label{dxdy1}\\
\dot y &= y\Big[b - m_0 - (a+b)x - by - dy^{\delta}\Big] =: yG(x,y)\,. \label{dxdy2}
\end{align}
\end{subequations}
We analyze~\eqref{dxdy} throughout, denoting partial differentiation by $\partial$, with a subscript indicating the variable, and the derivative of a function of a single variable by a prime.

\smallskip

\subsection{Invariance and absence of periodic orbits}\label{subsec:invariance}

We work on the biologically relevant region
\begin{equation}\label{Omega}
\Omega:=\big\{(x,y)\in\R^2 \,:\, x>0,\ y>0,\ x+y<1\big\}\,,
\end{equation}
which is open and simply connected. The next result shows that $\Omega$ is positively invariant and contains no periodic orbits.

\begin{lem}\label{lem:dulac}
Assume~\eqref{GA}. Then the region $\Omega$ is positively invariant for~\eqref{dxdy} and contains no periodic orbit of the system.
\end{lem}

\begin{proof}
{\bf (i)} Fix $\big(x(0),y(0)\big)\in\Omega$. Since $\dot x=xF(x,y)$ and $\dot y=yG(x,y)$, the solution admits the representation
\[
x(t)=x(0)\exp\Big(\int_0^t F(x(s),y(s))\,\rd s\Big)\,,\quad
y(t)=y(0)\exp\Big(\int_0^t G(x(s),y(s))\,\rd s\Big)\,,
\]
so that $x(t)>0$ and $y(t)>0$ for all $t\ge0$. It thus remains to rule out exit through the segment $\{x+y=1\}$. Setting $z:=x+y$ and using $x=1-y$ on this segment, we obtain
\begin{equation}\label{zdot}
\dot z=\dot x+\dot y=-g-nx-m_0y-dy^{\delta+1}<0\,,
\end{equation}
since $g,n,m_0,d>0$ and $x,y>0$. Hence the trajectory cannot reach $\{x+y=1\}$, and $\Omega$ is positively invariant.

{\bf (ii)} The vector field $\mathbf V:=(xF,\,yG)$ and the rescaled field $\psi\mathbf V$ introduced below are continuously differentiable away from $\{x=0\}$, $\{y=0\}$, and $\{y=1\}$. As none of these sets meets~$\Omega$, both are of class $\mathrm{C}^1(\Omega)$ for every $\delta>0$. In particular $\mathbf V$ is locally Lipschitz on $\Omega$, so by the Picard--Lindel\"of theorem solutions with initial data in $\Omega$ exist and are unique. Choosing the Dulac function $\psi(x,y):=\tfrac{1}{xy}>0$, we obtain
\[
\partial_x\big(\psi x F\big)+\partial_y\big(\psi y G\big)
=\frac{\partial_x F}{y}+\frac{\partial_y G}{x}
=-\frac{\gamma}{y}-\frac{b+\delta dy^{\delta-1}}{x}<0
\qquad\text{on }\Omega\,.
\]
Since $\Omega$ is simply connected and this expression has a strict sign, the Bendixson--Dulac criterion rules out periodic orbits in~$\Omega$.
\end{proof}

In view of Lemma~\ref{lem:dulac}, we assume henceforth that the initial data satisfies $(x(0),y(0))\in\Omega$.

\bigskip

\section{Boundary Equilibria}\label{sec:boundary}

In this section we locate the boundary equilibria of~\eqref{dxdy} and determine their local stability. Equilibria are the intersections of the nullclines, which we therefore describe first. As we shall see, the grazing intensity~$g$ governs the stability of the non-trivial boundary equilibria through two thresholds $0<g_0<g_1<\gamma-n$ that separate a low-, an intermediate-, and a high-grazing regime.

\smallskip
\subsection{Nullclines}\label{subsec:nullclines}

The equilibria of~\eqref{dxdy} are the points of $\overline\Omega$ at which $\dot x=\dot y=0$, that is, the intersections of the {\it macroalgae nullcline} $\{x=0\}\cup\{F=0\}$ with the {\it coral nullcline} $\{y=0\}\cup\{G=0\}$. These nullclines and the equilibria at their intersections are illustrated, for several grazing rates, in Figure~\ref{fig:phaseplane}.


\medskip
\noindent\textit{The macroalgae nullcline.} Besides the $y$-axis $\{x=0\}$, the macroalgae nullcline is the zero set of $F$. Since $a<\gamma$, multiplying $F(x,y)=0$ by $1-y>0$ yields a quadratic in $y$ whose branch lying in $\overline\Omega$ is the graph
\begin{equation}\label{f}
y=M(x):=\frac{1}{2(\gamma-a)}\Big[2\gamma-a-n-\gamma x-\sqrt{(\gamma x+n-a)^2+4g(\gamma-a)}\Big]\,.
\end{equation}
The expression under the square root is bounded below by $4g(\gamma-a)>0$, so $M$ is smooth on $[0,\infty)$. 
The curve joins the point
\begin{equation}\label{yA}
A=(0,y_A)\,,\qquad y_A:=M(0)=\frac{2\gamma-a-n-\sqrt{(n-a)^2+4g(\gamma-a)}}{2(\gamma-a)}\,,
\end{equation}
on the $y$-axis to the point $B=(x_B,0)$ on the $x$-axis, where
\begin{equation}\label{xB}
x_B:=1-\frac{g+n}{\gamma}\,,
\end{equation}
and under~\eqref{GA} both $y_A\in(0,1)$ and $x_B\in(0,1)$. Differentiating~\eqref{f} gives
\begin{equation}\label{fprime}
M'(x)=-\frac{\gamma}{2(\gamma-a)}\Bigg[1+\frac{\gamma x+n-a}{\sqrt{(\gamma x+n-a)^2+4g(\gamma-a)}}\Bigg]<0\,,
\end{equation}
so $M$ is strictly decreasing on $[0,x_B]$.

\medskip
\noindent\textit{The coral nullcline.} Besides the $x$-axis $\{y=0\}$, the coral nullcline is the zero set of $G$. As $G$ is affine in $x$, the equation $G(x,y)=0$ is solved explicitly for $x$:
\begin{equation}\label{K}
x=K(y):=\frac{b-m_0-by-dy^{\delta}}{a+b}\,,\qquad y\in[0,y_C]\,,
\end{equation}
where $y_C\in(0,1)$ is the unique solution of
\begin{equation}\label{yC}
H(y_C)=b-m_0\,,\qquad H(y):=by+dy^{\delta}\,.
\end{equation}
Indeed, $H$ is continuous and strictly increasing on $[0,\infty)$ with $H(0)=0<b-m_0<b+d=H(1)$, so~\eqref{yC} admits a unique root $y_C\in(0,1)$. The curve $K$ joins $C=(0,y_C)$ on the $y$-axis to $D=(x_D,0)$ on the $x$-axis, where
\begin{equation}\label{xD}
x_D:=K(0)=\frac{b-m_0}{a+b}\in(0,1)\,.
\end{equation}
For $y>0$ we have
\begin{equation}\label{Kprime}
K'(y)=-\frac{b+\delta dy^{\delta-1}}{a+b}<0\,,
\end{equation}
so $K$ is strictly decreasing and is smooth on $(0,y_C]$, extending as a curve of class $\mathrm{C}^1$ to $[0,y_C]$ when $\delta\ge1$. 
For later use we record that
\begin{equation}\label{partials}
\partial_yF=(a-\gamma)-\frac{g}{(1-y)^2}<0\qquad\text{and}\qquad \partial_yG=-(b+\delta dy^{\delta-1})<0\qquad\text{on }\Omega\,.
\end{equation}

The boundary equilibria are now readily identified.

\begin{lem}\label{lem:boundary}
Under~\eqref{GA}, system~\eqref{dxdy} has exactly three boundary equilibria: the trivial state $O=(0,0)$, the macroalgae-only state $B=(x_B,0)$ with $x_B\in(0,1)$ as in~\eqref{xB}, and the coral-only state $C=(0,y_C)$ with $y_C\in(0,1)$ the unique root of~\eqref{yC}.
\end{lem}

\begin{proof}
A boundary equilibrium lies on $\partial\Omega$. On the open segment $\{x+y=1\}$ we have
\[
G(x,1-x)=-m_0-ax-d(1-x)^{\delta}<0\,,
\]
so that $\dot y=yG<0$ whenever $y>0$, and this segment therefore carries no equilibrium. On the $x$-axis we have $\dot y=0$, while $\dot x=xF(x,0)$ with
\[
F(x,0)=\gamma(1-x)-n-g
\]
strictly decreasing and vanishing only at $x_B$, so that the equilibria on $\{y=0\}$ are $O$ and $B$. In the same way, on the $y$-axis we have $\dot x=0$, while $\dot y=yG(0,y)$ with
\[
G(0,y)=b-m_0-by-dy^{\delta}
\]
strictly decreasing and vanishing only at $y_C$, so that the equilibria on $\{x=0\}$ are $O$ and $C$.
\end{proof}

\smallskip

\subsection{Local stability of the boundary equilibria}\label{subsec:stab-boundary}

Linearizing~\eqref{dxdy}, and recalling~\eqref{partials} together with $\partial_xF=-\gamma$ and $\partial_xG=-(a+b)$, the Jacobian is
\begin{equation}\label{Jac}
J(x,y)=\begin{pmatrix} F+x\,\partial_xF & x\partial_yF\\[4pt] y\partial_xG & G+y\partial_yG\end{pmatrix}
=\begin{pmatrix} F-\gamma x & x\partial_yF\\[4pt] -(a+b)y & G+y\partial_yG\end{pmatrix}\,,
\end{equation}
where the entries are evaluated using $\partial_yF$ and $\partial_yG$ from~\eqref{partials}. We evaluate~\eqref{Jac} at each boundary equilibrium.

\begin{lem}\label{lem:stabO}
Under~\eqref{GA}, the trivial equilibrium $O=(0,0)$ is an unstable node.
\end{lem}

\begin{proof}
By~\eqref{Jac},
\[
J(0,0)=\begin{pmatrix}\gamma-n-g & 0\\[3pt] 0 & b-m_0\end{pmatrix}\,,
\]
whose eigenvalues $\gamma-n-g$ and $b-m_0$ are both positive by~\eqref{GA}.
\end{proof}

We next turn to the macroalgae-only state $B$, whose stability is set by the threshold
\begin{equation}\label{g1}
g_1:=\frac{\gamma(a+m_0)}{a+b}-n\,.
\end{equation}

\begin{thm}\label{thm:stabB}
Suppose~\eqref{GA} holds. Then $0<g_1<\gamma-n$, and for every $\delta>0$ the macroalgae-only state $B$ is a stable node when $g<g_1$, a saddle when $g>g_1$, and nonhyperbolic at $g=g_1$.
\end{thm}

\begin{rem}
The threshold $g_1$ is a bifurcation value: Section~\ref{sec:bif} shows that $B$ is then a saddle-node point at which the system undergoes a transcritical bifurcation as $g$ crosses $g_1$ (Theorem~\ref{thm:transcritB} and Proposition~\ref{prop:sn}).
\end{rem}

\begin{proof}
At $B=(x_B,0)$ the Jacobian~\eqref{Jac} is upper triangular,
\[
J(x_B,0)=\begin{pmatrix}-(\gamma-g-n) & x_B(a-\gamma-g)\\[3pt] 0 & -(a+m_0)+\dfrac{g+n}{\gamma}(a+b)\end{pmatrix}\,,
\]
so that its eigenvalues are the diagonal entries. By~\eqref{GA} the first is $\lambda_1=-(\gamma-g-n)<0$, while the second, by the definition~\eqref{g1} of $g_1$, satisfies
\[
\lambda_2=-(a+m_0)+\frac{g+n}{\gamma}(a+b)=\frac{a+b}{\gamma}(g-g_1)\,.
\]
Thus $\lambda_2<0$ for $g<g_1$, $\lambda_2=0$ for $g=g_1$, and $\lambda_2>0$ for $g>g_1$, so that $B$ is a stable node when $g<g_1$, nonhyperbolic when $g=g_1$, and a saddle when $g>g_1$. Finally $g_1<\gamma-n$ since $m_0<b$, and $g_1>0$ by~\eqref{GA}.
\end{proof}

The stability of the coral-only state $C$ is governed by the threshold
\begin{equation}\label{g0}
g_0:=(1-y_C)\big(\gamma-n+(a-\gamma)y_C\big)\,,
\end{equation}
with $y_C$ as in~\eqref{yC}.

\begin{thm}\label{thm:stabC}
Suppose~\eqref{GA} holds. Then $g_0>0$, and for every $\delta>0$ the coral-only state $C$ is a stable node when $g>g_0$, a saddle when $g<g_0$, and nonhyperbolic at $g=g_0$.
\end{thm}

\begin{rem}
The threshold $g_0$ is a bifurcation value: Section~\ref{sec:bif} shows that $C$ is then a saddle-node point at which the system undergoes a transcritical bifurcation as $g$ crosses $g_0$ (Theorem~\ref{thm:transcritC} and Proposition~\ref{prop:sn}).
\end{rem}

\begin{proof}
At $C=(0,y_C)$ the Jacobian~\eqref{Jac} is lower triangular,
\[
J(0,y_C)=\begin{pmatrix}\mu_1 & 0\\[3pt] -(a+b)y_C & \mu_2\end{pmatrix}\,,
\]
so that its eigenvalues are the diagonal entries $\mu_1$ and $\mu_2$. The lower-right entry simplifies by~\eqref{yC} from $b-m_0-2by_C-(1+\delta)dy_C^{\delta}$ to
\[
\mu_2=-by_C-\delta dy_C^{\delta}<0\,,
\]
while the upper-left entry, by the definition~\eqref{g0} of $g_0$, satisfies
\[
\mu_1=F(0,y_C)=\gamma-n+(a-\gamma)y_C-\frac{g}{1-y_C}=\frac{g_0-g}{1-y_C}\,.
\]
Since $1-y_C>0$, we have $\mu_1>0$ for $g<g_0$, $\mu_1=0$ for $g=g_0$, and $\mu_1<0$ for $g>g_0$, so that $C$ is a saddle when $g<g_0$, nonhyperbolic when $g=g_0$, and a stable node when $g>g_0$. Moreover $\gamma-n+(a-\gamma)y_C>a-n>0$ since $a>n$, so that $g_0>0$.
\end{proof}

The two thresholds are ordered, which is what separates the grazing regimes studied in Section~\ref{sec:interior}.

\begin{lem}\label{lem:ordering}
Under~\eqref{GA}, the thresholds satisfy $0<g_0<g_1<\gamma-n$.
\end{lem}

\begin{proof}
We have already established $g_0>0$ and $0<g_1<\gamma-n$, so it remains to prove $g_0<g_1$. Since $0<y_C<1$ and $\delta>0$ give $0<y_C^{\delta}<1$, relation~\eqref{yC} yields $by_C=b-m_0-dy_C^{\delta}>b-m_0-d$, that is,
\[
1-y_C<\frac{m_0+d}{b}\,.
\]
Using this lower bound for $y_C$ once more,
\[
\gamma-n+(a-\gamma)y_C=\gamma-n-(\gamma-a)y_C<(a-n)+(\gamma-a)\frac{m_0+d}{b}\,,
\]
and therefore, by~\eqref{g0},
\begin{equation}\label{g0bound}
g_0<\frac{m_0+d}{b}\left[(a-n)+(\gamma-a)\frac{m_0+d}{b}\right]
=\frac{(a-n)(m_0+d)}{b}+\frac{(\gamma-a)(m_0+d)^2}{b^2}\,.
\end{equation}
By the last condition in~\eqref{GA}, $m_0+d<\tfrac{b}{2(a-n)}g_1$ and $m_0+d<b\sqrt{\tfrac{g_1}{2(\gamma-a)}}$, so each summand in~\eqref{g0bound} is strictly less than $\tfrac{g_1}{2}$. Hence $g_0<g_1$.
\end{proof}

\begin{rem}\label{rem:thresholds}
The two thresholds control the boundary stability monotonically in $g$. The proofs of Theorems~\ref{thm:stabB} and~\ref{thm:stabC} give
\[
\lambda_2=\frac{a+b}{\gamma}(g-g_1),\qquad
\mu_1=\frac{g_0-g}{1-y_C},
\]
hence
\[
\partial_g\lambda_2=\frac{a+b}{\gamma}>0,\qquad
\partial_g\mu_1=-\frac{1}{1-y_C}<0 .
\]
As $g$ increases, $\lambda_2$ therefore increases through $0$ at $g=g_1$ and $\mu_1$ decreases through $0$ at $g=g_0$, so stronger grazing destabilizes $B$ and stabilizes $C$. Since $g_0<g_1$ by Lemma~\ref{lem:ordering}, the coral state gains stability before the macroalgae state loses it, and $B$ and $C$ are simultaneously stable for $g_0<g<g_1$. In terms of the nullcline intercepts~\eqref{yA}--\eqref{xB} and~\eqref{xD},
\[
g>g_0 \quad\text{if and only if}\quad y_C>y_A,
\qquad
g<g_1 \quad\text{if and only if}\quad x_D<x_B .
\]
\end{rem}

\bigskip

\section{Interior Equilibria}\label{sec:interior}

An interior equilibrium is a point $(x^\star,y^\star)\in\Omega$ at which $F=G=0$, and it corresponds to the coexistence of coral, macroalgae, and turf, the last occupying the remaining cover $1-x^\star-y^\star>0$. In this section we reduce the search for such points to a scalar equation, relate the stability of an interior equilibrium to the derivative of the reduced function $L$ introduced below, and classify the dynamics in each grazing regime.

\smallskip

\subsection{The reduced equation and its concavity}\label{subsec:reduced}

By~\eqref{K}, the coral nullcline is the graph $x=K(y)$, $y\in[0,y_C]$, so every interior equilibrium has the form $(K(y^\star),y^\star)$ with $y^\star\in(0,y_C)$ a root of $F\big(K(y),y\big)$. Since $1-K(y)=\tfrac{a+m_0+by+dy^{\delta}}{a+b}$, we set
\begin{equation}\label{Ldef}
L(y):=F\big(K(y),y\big)=\frac{\gamma(a+m_0+by+dy^{\delta})}{a+b}-n+(a-\gamma)y-\frac{g}{1-y}\,,\qquad y\in[0,y_C]\,.
\end{equation}
Every root of $L$ in $(0,y_C)$ indeed yields an interior equilibrium.

\begin{lem}\label{lem:inOmega}
If $y\in(0,y_C)$, then $(K(y),y)\in\Omega$.
\end{lem}

\begin{proof}
Recall from~\eqref{yC} the strictly increasing function $H(y)=by+dy^{\delta}$. Since $H(y)<H(y_C)=b-m_0$ for $y<y_C$, we have $K(y)>0$. Moreover $K(y)+y=\tfrac{b-m_0+ay-dy^{\delta}}{a+b}<1$, since $ay-dy^{\delta}<a<a+m_0$ for $0<y<1$. Thus $(K(y),y)\in\Omega$.
\end{proof}

Consequently the number of interior equilibria equals the number of roots of $L$ in $(0,y_C)$. Evaluating $L$ at the endpoints and using $a+m_0+by_C+dy_C^{\delta}=a+b$ from~\eqref{yC},
\begin{equation}\label{Lends}
L(0)=g_1-g\,,\qquad L(y_C)=\frac{g_0-g}{1-y_C}\,,
\end{equation}
with $g_0,g_1$ from~\eqref{g0} and~\eqref{g1}. The number of roots is determined by the shape of $L$, whose derivatives are
\begin{equation}\label{Lderivs}
L'(y)=\frac{\gamma(b+\delta d\,y^{\delta-1})}{a+b}+(a-\gamma)-\frac{g}{(1-y)^2}\,,\quad
L''(y)=\frac{\gamma\delta(\delta-1)dy^{\delta-2}}{a+b}-\frac{2g}{(1-y)^3}\,,
\end{equation}
for $y\in(0,y_C]$. We record when $L$ is concave.

\begin{lem}\label{lem:concavity}
Let $y\in(0,y_C]$. {\bf (i)} If $0<\delta\le1$, then $L''(y)<0$. {\bf (ii)} If $\delta\ge2$ and
\begin{equation}\label{concbound}
2g\ge\frac{\gamma\delta(\delta-1)d}{a+b}y_C^{\delta-2}\,,
\end{equation}
then $L''(y)<0$. In either case $L$ is strictly concave on $(0,y_C]$.
\end{lem}

\begin{proof}
{\bf (i)} For $0<\delta\le1$ one has $\delta(\delta-1)\le0$, so the first term of $L''$ is nonpositive while the second is negative. {\bf (ii)} For $\delta\ge2$ the map $y\mapsto y^{\delta-2}$ is nondecreasing, so 
\[
\tfrac{\gamma\delta(\delta-1)d}{a+b}y^{\delta-2}\le\tfrac{\gamma\delta(\delta-1)d}{a+b}y_C^{\delta-2}\le 2g
\]
by~\eqref{concbound}, while $\tfrac{2g}{(1-y)^3}>2g$ for $y>0$, and hence $L''(y)<0$. Finally, since $y_C^{\delta-2}\le1$, the simpler bound $2g\ge\tfrac{\gamma\delta(\delta-1)d}{a+b}$ already implies~\eqref{concbound}, so it too suffices.
\end{proof}

The remaining range $1<\delta<2$, where $L$ fails to be concave near $y=0$, is analyzed in Subsection~\ref{subsec:nonconcave}.

\smallskip

\subsection{Stability of an interior equilibrium}\label{subsec:int-stab}

Let $E^\star=(x^\star,y^\star)$ with $x^\star=K(y^\star)$ be an interior equilibrium, $y^\star\in(0,y_C)$ a root of $L$. As $F(E^\star)=G(E^\star)=0$, the Jacobian~\eqref{Jac} reduces to
\begin{equation}\label{eq:Jacobi_iE}
J(E^\star)=\begin{pmatrix} x^\star\partial_xF(E^\star) & x^\star\partial_yF(E^\star)\\[3pt] y^\star\partial_xG(E^\star) & y^\star\partial_yG(E^\star)\end{pmatrix}\,.
\end{equation}
Its trace is
\begin{equation}\label{trJ}
\operatorname{tr}J(E^\star)=x^\star\partial_xF(E^\star)+y^\star\partial_yG(E^\star)=-\gamma x^\star-y^\star\big(b+\delta d(y^\star)^{\delta-1}\big)<0\,.
\end{equation}
For the determinant, the chain rule applied to~\eqref{Ldef} gives
\begin{equation}\label{Lstar}
L'(y^\star)=\partial_xF(E^\star)K'(y^\star)+\partial_yF(E^\star)=\frac{\gamma\big(b+\delta d(y^\star)^{\delta-1}\big)}{a+b}+\partial_yF(E^\star)\,,
\end{equation}
so that $\big(\partial_xF\partial_yG-\partial_yF\partial_xG\big)(E^\star)=\gamma\big(b+\delta d(y^\star)^{\delta-1}\big)+(a+b)\partial_yF(E^\star)=(a+b)L'(y^\star)$, and therefore
\begin{equation}\label{detJ}
\det J(E^\star)=x^\star y^\star\big(\partial_xF\partial_yG-\partial_yF\partial_xG\big)(E^\star)=(a+b)x^\star y^\star L'(y^\star)\,.
\end{equation}
The off-diagonal entries $x^{\star}\partial_yF(E^\star)$ and $y^\star\partial_xG(E^\star)=-(a+b)y^\star$ are both negative, so that their product is positive and the eigenvalues of $J(E^\star)$ are real. In view of~\eqref{trJ} and~\eqref{detJ}, the equilibrium $E^\star$ is therefore a stable node when $L'(y^\star)>0$ and a saddle when $L'(y^\star)<0$, so that the sign of $L'$ at a root decides stability and the whole interior dynamics is read off from the graph of $L$.

\smallskip

\subsection{Classification of the dynamics}\label{subsec:classification}

It is convenient to write $L(y)=\dfrac{\tilde g(y)-g}{1-y}$,
where
\begin{equation}\label{tildeg}
\tilde g(y):=(1-y)\Big[\tfrac{\gamma}{a+b}\big(a+m_0+by+dy^{\delta}\big)-n+(a-\gamma)y\Big]\,,\qquad y\in[0,y_C]\,,
\end{equation}
is independent of $g$. Then $\tilde g(0)=g_1$ and $\tilde g(y_C)=g_0$, and a root of $L$ is a point where $\tilde g(y)=g$, while $L'(y^\star)$ and $\tilde g'(y^\star)$ share the same sign there, so that an interior equilibrium is a stable node when $\tilde g'(y^\star)>0$ and a saddle when $\tilde g'(y^\star)<0$. Set
\begin{equation}\label{gstar}
g^\star:=\max_{y\in[0,y_C]}\tilde g(y)\in[g_1,\gamma-n)\,.
\end{equation}
The complete picture is the following.

\begin{thm}\label{thm:classification}
Suppose~\eqref{GA} holds and $L$ is strictly concave on $(0,y_C]$ (e.g.\ $0<\delta\le1$, or $\delta\ge2$ with~\eqref{concbound}). The interior equilibria of~\eqref{dxdy} are classified by the grazing regime:
\begin{itemize}
\item[\textup{(i)}] \emph{Low grazing.} For $0<g<g_0$ there is no interior equilibrium.
\item[\textup{(ii)}] \emph{Intermediate grazing.} For $g_0<g<g_1$ there is a unique interior equilibrium $E^\star=(K(y^\star),y^\star)$, and it is a saddle.
\item[\textup{(iii)}] \emph{High grazing.} Let $g_1<g<\gamma-n$. If $g^\star=g_1$, there is no interior equilibrium. If $g^\star>g_1$, there is none for $g^\star<g<\gamma-n$, exactly two for $g_1<g<g^\star$, a stable node $E_1^\star=(K(y_1^\star),y_1^\star)$ and a saddle $E_2^\star=(K(y_2^\star),y_2^\star)$ with $0<y_1^\star<y_2^\star<y_C$, and a single nonhyperbolic equilibrium $E_3^\star=(K(y_3^\star),y_3^\star)$ at $g=g^\star$.
\end{itemize}
\end{thm}

\begin{proof}
{\bf (i)} By~\eqref{Lends}, $L(0)>0$ and $L(y_C)>0$ (as $g<g_0$). Writing $y=(1-\theta)y_C$ with $\theta\in[0,1]$, concavity gives
\[
L(y)\ge\theta L(0)+(1-\theta)L(y_C)>0\,,
\]
so that $L>0$ on $[0,y_C]$ and there is no root.

{\bf (ii)} By~\eqref{Lends}, $L(0)>0>L(y_C)$ (as $g_0<g<g_1$), so $L$ has a root $y^\star\in(0,y_C)$. If $0<y_1<y_2<y_C$ were two roots, then writing $y_1=(1-\theta)y_2$ with $\theta\in(0,1)$, concavity would give
\[
0=L(y_1)\ge\theta L(0)+(1-\theta)L(y_2)=\theta L(0)>0\,,
\]
a contradiction, so that the root is unique. Finally, by the Mean Value Theorem there is some $c\in(0,y^\star)$ with
\[
L'(c)=\frac{L(y^\star)-L(0)}{y^\star}=-\frac{L(0)}{y^\star}<0\,,
\]
and since $L'$ is decreasing we have $L'(y^\star)\le L'(c)<0$, so that the corresponding equilibrium $E^\star$ is a saddle.

{\bf (iii)} By~\eqref{Lends}, $L(0)<0$ and $L(y_C)<0$ (as $g>g_1$). If $g^\star=g_1$, then $\tilde g\le g_1<g$, hence $L<0$ on $[0,y_C]$ and there is no interior equilibrium. Assume henceforth $g^\star>g_1$, so that the maximum of $\tilde g$ is attained at an interior point $\hat y\in(0,y_C)$, as $\tilde g(0)=g_1<g^\star$ and $\tilde g(y_C)=g_0<g^\star$. We distinguish three cases.

\emph{(a)} For $g^\star<g<\gamma-n$ one has $\tilde g<g$, hence $L<0$ and there is no interior equilibrium.

\emph{(b)} For $g_1<g<g^\star$,
\[
L(\hat y)=\frac{g^\star-g}{1-\hat y}>0\,,\qquad L(0)<0\,,\qquad L(y_C)<0\,,
\]
so $L$ has roots $y_1^\star\in(0,\hat y)$ and $y_2^\star\in(\hat y,y_C)$. Strict concavity precludes a third root, since three roots $y_1<y_2<y_3$ would force $L(y_2)>0$ by the argument of part~(ii). As $L>0$ on $(y_1^\star,y_2^\star)$, vanishes at the endpoints, and $L'$ is strictly decreasing,
\[
L'(y_1^\star)>0>L'(y_2^\star)\,,
\]
so the corresponding equilibria are a stable node $E_1^\star$ and a saddle $E_2^\star$.

\emph{(c)} For $g=g^\star$, $L\le0$ with equality only at the unique maximizer $y_3^\star$ of $\tilde g$ (unique by strict concavity), where
\[
L(y_3^\star)=0\,,\qquad L'(y_3^\star)=0\,,
\]
the second equality because $y_3^\star$ is an interior maximum. By~\eqref{detJ} and~\eqref{trJ},
\[
\det J(E_3^\star)=0\,,\qquad \operatorname{tr}J(E_3^\star)<0\,.
\]
\end{proof}

Together with Lemma~\ref{lem:stabO} and Theorems~\ref{thm:stabB} and~\ref{thm:stabC}, this yields the  classification summarized in Table~\ref{tab:classification} and illustrated in Figure~\ref{fig:phaseplane}. When $g^\star>g_1$, Section~\ref{sec:bif} shows that the coalescence of $E_1^\star$ and $E_2^\star$ into $E_3^\star$ at $g=g^\star$ is a saddle-node bifurcation (Theorem~\ref{thm:saddlenode} and Proposition~\ref{prop:sn}), illustrated in Figure~\ref{fig:twointerior}. 

\begin{table}[H]
\centering
\small
\setlength{\tabcolsep}{4pt}
\renewcommand{\arraystretch}{1.25}
\begin{tabular}{lcccl}
\toprule
& $O$ & $B$ & $C$ & interior equilibria\\
\midrule
$0<g<g_0$        & unstable node & stable node   & saddle       & none\\
$g=g_0$          & unstable node & stable node   & nonhyperbolic & none\\
$g_0<g<g_1$      & unstable node & stable node   & stable node  & $E^\star$ saddle\\
$g=g_1$          & unstable node & nonhyperbolic & stable node  & $E_2^\star$ saddle\\
$g_1<g<g^\star$   & unstable node & saddle        & stable node  & $E_1^\star$ stable node, $E_2^\star$ saddle\\
$g=g^\star$       & unstable node & saddle        & stable node  & $E_3^\star$ nonhyperbolic\\
$g^\star<g<\gamma-n$ & unstable node & saddle     & stable node  & none\\
\bottomrule
\end{tabular}
\caption{Classification of equilibria across the grazing regimes (Lemma~\ref{lem:stabO} and Theorems~\ref{thm:stabB}, \ref{thm:stabC}, and~\ref{thm:classification}), where the last three rows assume $g^\star>g_1$, since otherwise there is no interior equilibrium for $g>g_1$. At $g=g_0$ and $g=g_1$ the equilibria $C$ and $B$ are nonhyperbolic and undergo transcritical bifurcations, and at $g=g^\star$ the interior equilibria undergo a saddle-node bifurcation (Section~\ref{sec:bif}). For $\delta \in (1,2)$ the counts can differ (Proposition~\ref{prop:nonconcave}).}
\label{tab:classification}
\end{table}

The two canonical crowding exponents, linear ($\delta=1$) and quadratic ($\delta=2$), are worked out explicitly in Subsection~\ref{sec:examples}.

\smallskip

\subsection{The non-concave regime \texorpdfstring{$1<\delta<2$}{1<delta<2}}\label{subsec:nonconcave}

When $1<\delta<2$, \eqref{Lderivs} gives
\[
\lim_{y\to0^+}L''(y)=+\infty\,,
\]
so that $L$ is convex near $y=0$ and the concavity arguments of Subsection~\ref{subsec:classification} no longer apply. We show that $L'''<0$ on $(0,y_C)$ (Lemma~\ref{lem:shape-nonconcave}), which limits the system to at most three interior equilibria and preserves the high-grazing conclusions of Theorem~\ref{thm:classification} (Proposition~\ref{prop:nonconcave}).

\begin{lem}\label{lem:shape-nonconcave}
Assume~\eqref{GA} and $1<\delta<2$. Then $L'''(y)<0$ and $\tilde g'''(y)<0$ for $y\in(0,y_C)$. Consequently $L'$ has at most one interior maximum, hence at most two zeros, and $\tilde g$ has at most one interior local maximum and one interior local minimum.
\end{lem}

\begin{proof}
By~\eqref{Lderivs},
\[
L'''(y)=\frac{\gamma\delta(\delta-1)(\delta-2)d}{a+b}y^{\delta-3}-\frac{6g}{(1-y)^4}\,,
\]
and for $1<\delta<2$ both terms are negative, since $\delta(\delta-1)>0$ and $\delta-2<0$, and hence $L'''<0$. Moreover $\tilde g(y)=(1-y)L(y)+g$, so that $\tilde g'''=-3L''+(1-y)L'''$, and substituting~\eqref{Lderivs} gives
\[
\tilde g'''(y)=\frac{\gamma\delta(\delta-1)d}{a+b}y^{\delta-3}\big[(\delta-2)-(\delta+1)y\big]<0\,,\qquad y\in(0,y_C)\,,
\]
the bracket being negative because $\delta-2<0$. As $L'''<0$, $L''$ is strictly decreasing, and since $L''(0^+)=+\infty$ it has at most one zero in $(0,y_C)$, so $L'$ has at most one interior maximum and hence at most two zeros. As $\tilde g'''<0$, $\tilde g'$ is strictly concave and so has at most two zeros, whence $\tilde g$ has at most one interior local maximum and one interior local minimum.
\end{proof}

\begin{prop}\label{prop:nonconcave}
Assume~\eqref{GA} and $1<\delta<2$. Then~\eqref{dxdy} has at most three interior equilibria $(K(y^\star),y^\star)$, each a stable node or a saddle according as $L'(y^\star)>0$ or $L'(y^\star)<0$ (Subsection~\ref{subsec:int-stab}). For $g>g_1$ there are at most two, and when two exist the one of smaller coral cover is a stable node and the other a saddle. The number of simple interior equilibria is even for $g<g_0$ and for $g>g_1$, and odd for $g_0<g<g_1$.
\end{prop}

\begin{proof}
{\bf (a)} By Lemma~\ref{lem:shape-nonconcave}, $L'$ has at most two zeros in $(0,y_C)$, so by Rolle's theorem $L$ has at most three roots there. As the interior equilibria are exactly the points $(K(y^\star),y^\star)$ with $y^\star\in(0,y_C)$ a root of $L$ (Subsection~\ref{subsec:reduced} and Lemma~\ref{lem:inOmega}), there are at most three of them.

{\bf (b)} For $g>g_1$, \eqref{Lends} gives $L(0)=g_1-g<0$. Were $L$ to have three roots $y_1<y_2<y_3$, Rolle would yield two zeros of $L'$, which by Lemma~\ref{lem:shape-nonconcave} are its only zeros. As $L'$ has a single interior maximum, $L'<0$ on $(0,y_1)$, whence $L(0)>L(y_1)=0$, contradicting $L(0)<0$. Hence at most two roots remain. For two simple roots $y_1<y_2$ the endpoint signs $L(0),L(y_C)<0$ force
\[
L<0 \ \text{ on }\ (0,y_1)\cup(y_2,y_C)\,,\qquad L>0 \ \text{ on }\ (y_1,y_2)\,,
\]
so that
\[
L'(y_1)>0>L'(y_2)\,,
\]
and hence $(K(y_1),y_1)$ is a stable node and $(K(y_2),y_2)$ a saddle.

{\bf (c)} By~\eqref{Lends},
\[
\sign L(0)=\sign(g_1-g)\,,\qquad \sign L(y_C)=\sign(g_0-g)\,.
\]
A continuous function has an even or odd number of sign-changing zeros depending on whether its endpoint values agree or differ in sign. Since $g_0<g_1$, these agree for $g<g_0$, differ for $g_0<g<g_1$, and agree for $g>g_1$, so the number of simple interior equilibria is even, odd, and even in these three regimes.
\end{proof}

For $g_0<g<g_1$ there can be three interior equilibria, a tristable regime illustrated in Figure~\ref{fig:threeinterior}.

\begin{rem}\label{rem:nonconcave}
By Proposition~\ref{prop:nonconcave}(a),(b), the high-grazing conclusion of Theorem~\ref{thm:classification}(iii) remains valid for $1<\delta<2$: for $g>g_1$ there are at most two interior equilibria, a stable node and a saddle. For smaller $g$ the convexity of $L$ near $y=0$ can give the strictly concave $\tilde g'$ two interior zeros, so $\tilde g$ has an interior local minimum, and a second pair of interior equilibria can raise the count to three over a narrow range of $g$.
When $\tilde g$ is monotone on $[0,y_C]$, it can only decrease, since $\tilde g(0)=g_1$ exceeds $\tilde g(y_C)=g_0$ (Lemma~\ref{lem:ordering}), so $g^\star=g_1$, leaving at most one interior equilibrium. A parameter set realizing three coexisting interior equilibria is examined in Subsection~\ref{sec:examples}.
\end{rem}

\bigskip

\section{Bifurcation Analysis}\label{sec:bif}
We now describe the bifurcations through which the equilibria of~\eqref{dxdy} exchange stability and are created or destroyed as the grazing intensity $g$ varies. The bifurcation diagrams of Figures~\ref{fig:twointerior} and~\ref{fig:threeinterior} display these bifurcations. By Section~\ref{sec:boundary} the boundary equilibria $B$ and $C$ are nonhyperbolic at $g=g_1$ and $g=g_0$, and by Section~\ref{sec:interior} the two interior equilibria of the high-grazing regime coalesce into the nonhyperbolic equilibrium $E_3^\star$ at $g=g^\star$. We show that the former undergo transcritical bifurcations and the latter a saddle-node bifurcation. We write system~\eqref{dxdy} as
\begin{equation}\label{Fmatrix}
\dot{\mathbf X}=\mathbf F(\mathbf X,g)\,,\qquad \mathbf X=(x,y)^\top\,,\quad \mathbf F(\mathbf X,g)=\begin{pmatrix}xF(x,y)\\ yG(x,y)\end{pmatrix}\,,
\end{equation}
and invoke Sotomayor's theorem~\cite[Sec.~4.2]{Pe01}: if $\mathbf F(\mathbf X_0,g^\dagger)=0$
and $D\mathbf F(\mathbf X_0,g^\dagger)$ has a simple zero eigenvalue with right eigenvector $\mathbf U$ and left eigenvector $\mathbf W$, then at $g=g^\dagger$ the system undergoes a {\it transcritical} bifurcation if
\begin{equation}\label{sotoT}
\mathbf W^\top\mathbf F_g(\mathbf X_0,g^\dagger)=0\,,\quad \mathbf W^\top\big(D\mathbf F_g(\mathbf X_0,g^\dagger)\mathbf U\big)\neq0\,,\quad \mathbf W^\top\big(D^2\mathbf F(\mathbf X_0,g^\dagger)(\mathbf U,\mathbf U)\big)\neq0\,,
\end{equation}
and a {\it saddle-node} bifurcation if
\begin{equation}\label{sotoSN}
\mathbf W^\top\mathbf F_g(\mathbf X_0,g^\dagger)\neq0\,,\quad \mathbf W^\top\big(D^2\mathbf F(\mathbf X_0,g^\dagger)(\mathbf U,\mathbf U)\big)\neq0\,.
\end{equation}
Here $\mathbf F_g=\partial_g\mathbf F$, $D\mathbf F_g$ is its Jacobian in $\mathbf X$, and $D^2\mathbf F(\mathbf U,\mathbf U)=\big(\mathbf U^\top(D^2F_1)\mathbf U, \mathbf U^\top(D^2F_2)\mathbf U\big)^\top$ with $D^2F_i$ the Hessian of the $i$-th component.

\subsection{Transcritical bifurcations at the boundary equilibria}\label{subsec:bif-boundary}

\begin{thm}\label{thm:transcritB}
Suppose~\eqref{GA} holds and $\delta\in\N$. If
\begin{equation}\label{ndB}
g_1\neq a-\gamma+\tfrac{\gamma}{a+b}\big(b+d\mathds 1_{\{\delta=1\}}\big)\,,
\end{equation}
then system~\eqref{dxdy} undergoes a transcritical bifurcation at $B$ as $g$ crosses $g_1$.
\end{thm}

\begin{proof}
Center $B$ at the origin by $X=x-x_B$, $Y=y$, where $x_B=1-\tfrac{g+n}{\gamma}$ depends on $g$, so that the origin is an equilibrium for every $g$. In these coordinates~\eqref{Fmatrix} reads $\dot{\mathbf X}=\mathbf F(\mathbf X,g)$ with $\mathbf X=(X,Y)^\top$ and
\begin{align*}
\mathbf F(\mathbf X,g)=\begin{pmatrix}F_1(\mathbf X,g)\\ F_2(\mathbf X,g)\end{pmatrix}
&=\begin{pmatrix}(X+x_B)\,F(X+x_B,Y)\\ Y\,G(X+x_B,Y)\end{pmatrix}\\[2pt]
&=\begin{pmatrix}(X+x_B)\big[\gamma(1-X-x_B)-n+(a-\gamma)Y-\tfrac{g}{1-Y}\big]\\[4pt]
Y\big[-(a+m_0)+\tfrac{(a+b)(g+n)}{\gamma}-(a+b)X-bY-dY^{\delta}\big]\end{pmatrix}\,,
\end{align*}
the lower entry written using this value of $x_B$.
For $\delta\in\N$ the component $F_2$ is a polynomial in $(X,Y)$, so $\mathbf F$ is smooth near the origin. At $g=g_1$ one has $\tfrac{(a+b)(g_1+n)}{\gamma}=a+m_0$, so
\[
D\mathbf F(\mathbf 0,g_1)=\begin{pmatrix}\lambda_1 & x_B(a-\gamma-g_1)\\ 0 & 0\end{pmatrix}\,,\qquad \lambda_1=-(\gamma-g_1-n)<0\,,
\]
and the zero eigenvalue is simple, with right and left eigenvectors
\[
\mathbf U=\Big(-\tfrac{x_B(a-\gamma-g_1)}{\lambda_1},1\Big)^\top\,,\qquad \mathbf W=(0,1)^\top\,.
\]
As $\mathbf W=(0,1)^\top$, the conditions~\eqref{sotoT} involve only $F_2$. From $\partial_g F_2(\mathbf X,g)=\tfrac{a+b}{\gamma}Y$,
\[
\mathbf W^\top\mathbf F_g(\mathbf 0,g_1)=0\,,\qquad
\mathbf W^\top\big(D\mathbf F_g(\mathbf 0,g_1)\mathbf U\big)=DF_{2,g}(\mathbf 0,g_1)\mathbf U=\tfrac{a+b}{\gamma}\neq0\,.
\]
Since $F_2(\cdot,\cdot,g_1)=-(a+b)XY-bY^2-dY^{\delta+1}$, its Hessian at the origin is
\[
D^2F_2(\mathbf 0,g_1)=\begin{pmatrix}0&-(a+b)\\-(a+b)&-2(b+d\mathds 1_{\{\delta=1\}})\end{pmatrix}\,,
\]
and therefore
\[
\mathbf W^\top\big(D^2\mathbf F(\mathbf 0,g_1)(\mathbf U,\mathbf U)\big)=\mathbf U^\top D^2F_2(\mathbf 0,g_1)\mathbf U=2\Big(\tfrac{(a+b)x_B(a-\gamma-g_1)}{\lambda_1}-\big(b+d\mathds 1_{\{\delta=1\}}\big)\Big)\,,
\]
which vanishes only for $g_1=a-\gamma+\tfrac{\gamma}{a+b}(b+d\mathds 1_{\{\delta=1\}})$, excluded by hypothesis. Hence~\eqref{sotoT} holds, and as the centering is a smooth change of coordinates, \eqref{dxdy} undergoes a transcritical bifurcation at $B$.
\end{proof}

\begin{rem}
The restriction $\delta\in\N$ guarantees that $Y^{\delta+1}$, and thus $F_2$, is smooth on a full neighborhood of $Y=0$, including $Y<0$, as Sotomayor's theorem requires, and the crowding coefficient $d$ enters the nondegeneracy condition~\eqref{ndB} only when $\delta=1$.
\end{rem}

\begin{thm}\label{thm:transcritC}
Suppose~\eqref{GA} holds. If
\begin{equation}\label{ndC}
g_0\neq(1-y_C)^2\Big(a-\gamma+\tfrac{\gamma(b+\delta dy_C^{\delta-1})}{a+b}\Big)\,,
\end{equation}
then for every $\delta>0$ system~\eqref{dxdy} undergoes a transcritical bifurcation at $C$ as $g$ crosses $g_0$.
\end{thm}

\begin{proof}
Since $C=(0,y_C)$ is independent of $g$, we work directly with~\eqref{Fmatrix}. At $g=g_0$ one has $\mu_1=F(0,y_C)=0$, so
\[
D\mathbf F(C,g_0)=\begin{pmatrix}0&0\\-(a+b)y_C&\mu_2\end{pmatrix}\,,\qquad \mu_2=-by_C-\delta dy_C^{\delta}<0\,,
\]
and the zero eigenvalue is simple, with
\[
\mathbf U=\Big(\tfrac{\mu_2}{(a+b)y_C},1\Big)^\top\,,\qquad \mathbf W=(1,0)^\top\,.
\]
As $\mathbf F_g(\mathbf X,g)=\big(-\tfrac{x}{1-y},0\big)^\top$ vanishes at $C$ (where $x=0$), $\mathbf W^\top\mathbf F_g(C,g_0)=0$, while
\[
\mathbf W^\top\big(D\mathbf F_g(C,g_0)\mathbf U\big)=-\tfrac{U_1}{1-y_C}=-\tfrac{\mu_2}{(a+b)y_C(1-y_C)}>0\,.
\]
Since $\mathbf W=(1,0)^\top$, the last condition involves only $F_1(\mathbf X,g)=xF(x,y)$, with Hessian
\[
D^2F_1(C,g_0)=\begin{pmatrix}-2\gamma & (a-\gamma)-\tfrac{g_0}{(1-y_C)^2}\\[2pt] (a-\gamma)-\tfrac{g_0}{(1-y_C)^2} & 0\end{pmatrix}\,,
\]
so that
\[
\mathbf W^\top\big(D^2\mathbf F(C,g_0)(\mathbf U,\mathbf U)\big)=\mathbf U^\top D^2F_1(C,g_0)\mathbf U=2U_1\Big(a-\gamma-\tfrac{g_0}{(1-y_C)^2}-\gamma U_1\Big)\,.
\]
Using $U_1=-\tfrac{b+\delta dy_C^{\delta-1}}{a+b}$, this vanishes only for $g_0=(1-y_C)^2\big(a-\gamma+\tfrac{\gamma(b+\delta dy_C^{\delta-1})}{a+b}\big)$, excluded by hypothesis. Thus~\eqref{sotoT} holds and \eqref{dxdy} undergoes a transcritical bifurcation at $C$.
\end{proof}

\begin{rem}
At $g=g_0$ both the macroalgal nullcline endpoint $A=(0,y_A)$ of~\eqref{yA} and the interior equilibrium $E^\star$ coincide with the coral-only state $C=(0,y_C)$: there $y_A=y_C$, since $g_0$ is the value at which $F(0,y_C)=0$, while $L(y_C)=0$ by~\eqref{Lends} and $K(y_C)=0$, so that $E^\star=(K(y_C),y_C)=(0,y_C)=C$. This crossing of the interior equilibrium with the boundary is the transcritical bifurcation of Theorem~\ref{thm:transcritC}.
\end{rem}

\smallskip 

\subsection{Saddle-node bifurcation of interior equilibria}\label{subsec:bif-interior}

\begin{thm}\label{thm:saddlenode}
Suppose~\eqref{GA} holds, that $L$ is strictly concave on $(0,y_C]$ (e.g.\ $0<\delta\le1$, or $\delta\ge2$ with~\eqref{concbound}), and that $g^\star>g_1$. Then for every $\delta>0$ system~\eqref{dxdy} undergoes a saddle-node bifurcation at the interior equilibrium $E_3^\star=(x_3^\star,y_3^\star)$ as $g$ crosses $g^\star$.
\end{thm}

\begin{proof}
By Theorem~\ref{thm:classification}(iii)(c), at $g=g^\star$ the equilibrium $E_3^\star$ satisfies $\det J(E_3^\star)=0$ and $\operatorname{tr}J(E_3^\star)<0$, so its zero eigenvalue is simple. Using $L'(y_3^\star)=0$, equivalently $\partial_yF(E_3^\star)=-\tfrac{\gamma(b+\delta d(y_3^\star)^{\delta-1})}{a+b}=\gamma U_1$, the eigenvectors are
\[
\mathbf U=\Big(-\tfrac{b+\delta d(y_3^\star)^{\delta-1}}{a+b},1\Big)^\top\,,\qquad \mathbf W=\Big(1,-\tfrac{\gamma x_3^\star}{(a+b)y_3^\star}\Big)^\top\,.
\]
Then $\mathbf W^\top\mathbf F_g(E_3^\star,g^\star)=-\tfrac{x_3^\star}{1-y_3^\star}\neq0$, so the bifurcation is not transcritical. For the second condition, the Hessians of $F_1(\mathbf X,g)=xF(x,y)$ and $F_2(\mathbf X,g)=yG(x,y)$ at $E_3^\star$ give, using $\partial_yF(E_3^\star)=\gamma U_1$ and $(a+b)U_1=-(b+\delta d(y_3^\star)^{\delta-1})$,
\[
\begin{aligned}
\mathbf U^\top D^2F_1(E_3^\star,g^\star)\mathbf U&=-\tfrac{2g^\star x_3^\star}{(1-y_3^\star)^3}\,,\\[2pt]
\mathbf U^\top D^2F_2(E_3^\star,g^\star)\mathbf U&=-2(a+b)U_1-2b-\delta(\delta+1)d(y_3^\star)^{\delta-1}\,.
\end{aligned}
\]
Combining these with $\mathbf W=(1,-\tfrac{\gamma x_3^\star}{(a+b)y_3^\star})^\top$ and simplifying by~\eqref{Lderivs},
\[
\mathbf W^\top\big(D^2\mathbf F(E_3^\star,g^\star)(\mathbf U,\mathbf U)\big)
=-\tfrac{2g^\star x_3^\star}{(1-y_3^\star)^3}+\tfrac{\gamma x_3^\star\delta(\delta-1)d(y_3^\star)^{\delta-2}}{a+b}
=x_3^\star L''(y_3^\star)\,.
\]
Since $L$ is strictly concave, $L''(y_3^\star)<0$, so this is nonzero. Thus~\eqref{sotoSN} holds and \eqref{dxdy} undergoes a saddle-node bifurcation at $E_3^\star$.
\end{proof}

Beyond the bifurcation type, the nondegeneracy conditions~\eqref{ndB} and~\eqref{ndC}, together with the strict concavity of $L$, fix the local phase portrait of each nonhyperbolic equilibrium at its threshold.

\begin{prop}\label{prop:sn}
Suppose~\eqref{GA} holds. Then each nonhyperbolic equilibrium is a saddle-node point in the sense of~\cite[Sec.~2.11]{Pe01}, as follows.
\begin{itemize}
\item[\textup{(i)}] The macroalgae-only state $B$ at $g=g_1$, provided $\delta\in\N$ and~\eqref{ndB} holds.
\item[\textup{(ii)}] The coral-only state $C$ at $g=g_0$, provided~\eqref{ndC} holds.
\item[\textup{(iii)}] The interior equilibrium $E_3^\star=(K(y_3^\star),y_3^\star)$ at $g=g^\star$, whenever $g^\star>g_1$ and $L$ is strictly concave on $(0,y_C]$.
\end{itemize}
\end{prop}

\begin{proof}
In each case the center-manifold reduction at the threshold has a leading term of order $m=2$, the normal form of a saddle-node. The computations are deferred to Appendix~\ref{app:proofs}.
\end{proof}

Together, Theorems~\ref{thm:transcritB},~\ref{thm:transcritC}, and~\ref{thm:saddlenode} show that, as the grazing intensity increases, the reef passes through a transcritical bifurcation at $g=g_0$ (where the coral-dominated state $C$ gains stability), a transcritical bifurcation at $g=g_1$ (where the macroalgae-dominated state $B$ loses stability), and, when $g^\star>g_1$, a saddle-node bifurcation at $g=g^\star$ (where the coexistence states annihilate one another). This full sequence of bifurcations, with the corresponding phase portraits, is depicted for a representative parameter set in Figure~\ref{fig:twointerior}. The ecological consequences are discussed in Section~\ref{sec:conclusion}.

\bigskip

\section{Examples and Numerical Simulations}\label{sec:exnum} 
\subsection{Special cases of the crowding exponent}\label{sec:examples}
We apply the foregoing analysis to three crowding regimes, linear ($\delta=1$), quadratic ($\delta=2$), and intermediate ($1<\delta<2$). For $\delta=1$ and $\delta=2$ the quantities of Section~\ref{sec:interior} become explicit: we specialize the coral nullcline~\eqref{K}, the reduced function~\eqref{Ldef} and its derivatives~\eqref{Lderivs}, the thresholds~\eqref{g0} and~\eqref{g1}, and the maximizer $\hat y$ of $\tilde g$ given in~\eqref{tildeg} that determines $g^\star$ from~\eqref{gstar}.\\

\noindent\emph{Linear crowding $\delta=1$.}\quad The nullcline~\eqref{K} is a line, $y_C$ from~\eqref{yC} is explicit, and~\eqref{g0} gives
\[
K(y)=\frac{b-m_0-(b+d)y}{a+b},\qquad y_C=\frac{b-m_0}{b+d}\,,\qquad
g_0=\frac{d+m_0}{b+d}\Big(\gamma-n+(a-\gamma)\tfrac{b-m_0}{b+d}\Big)\,.
\]
The $d$-term in~\eqref{Lderivs} vanishes, leaving $L''(y)=-\tfrac{2g}{(1-y)^3}<0$, so $L$ is strictly concave for every $g>0$ (Lemma~\ref{lem:concavity}(i)) and Theorem~\ref{thm:classification} applies. Theorems~\ref{thm:transcritB} and~\ref{thm:transcritC} (the former since $\delta\in\N$) give the transcritical bifurcations at $B$ and $C$ under~\eqref{ndB} and~\eqref{ndC}, which here read
\[
g_1\neq a-\gamma+\tfrac{\gamma(b+d)}{a+b},\qquad
g_0\neq(1-y_C)^2\Big(a-\gamma+\tfrac{\gamma(b+d)}{a+b}\Big)\,.
\]
Here $\tilde g$~\eqref{tildeg} is the quadratic $\tilde g(y)=(1-y)\big(g_1+(\tfrac{\gamma(b+d)}{a+b}+a-\gamma)y\big)$, with maximizer $\hat y$~\eqref{gstar} the root of $\tilde g'(y)=0$,
\[
\hat y=\frac{1}{2}\left(1-\frac{(a+b)g_1}{\gamma d+a(a+b-\gamma)}\right)\,,
\]
so that $g^\star=\tilde g(\hat y)$, and Theorem~\ref{thm:saddlenode} yields a saddle-node at $g^\star$ when $g^\star>g_1$, as in Figure~\ref{fig:twointerior} (a $\delta=1$ set with two interior equilibria by Theorem~\ref{thm:classification}(iii)).\\

\noindent\emph{Quadratic crowding $\delta=2$.}\quad The nullcline~\eqref{K} is a parabola and~\eqref{yC} gives
\[
K(y)=\frac{b-m_0-by-dy^2}{a+b},\qquad y_C=\frac{-b+\sqrt{b^2+4d(b-m_0)}}{2d}\,,
\]
with $g_0$ from~\eqref{g0}. Now~\eqref{Lderivs} gives $L''(y)=\tfrac{2\gamma d}{a+b}-\tfrac{2g}{(1-y)^3}$, so $L$ is strictly concave on $(0,y_C]$ exactly under the bound~\eqref{concbound}, that is $g\ge\tfrac{\gamma d}{a+b}$ (Lemma~\ref{lem:concavity}(ii)), and then Theorem~\ref{thm:classification} applies, as Figure~\ref{fig:bifdiagram} illustrates for $\delta=2$. Theorems~\ref{thm:transcritB} and~\ref{thm:transcritC} (the former since $\delta\in\N$) give the transcritical bifurcations at $B$ and $C$ under~\eqref{ndB} and~\eqref{ndC}, which here read
\[
g_1\neq a-\gamma+\tfrac{\gamma b}{a+b},\qquad
g_0\neq(1-y_C)^2\Big(a-\gamma+\tfrac{\gamma(b+2d y_C)}{a+b}\Big)\,.
\]
Here $\tilde g$~\eqref{tildeg} is cubic, with maximizer $\hat y$~\eqref{gstar} the larger root of $\tilde g'(y)=0$,
\[
\hat y=\frac{a+b}{3\gamma d}\left[-\Big(\tfrac{\gamma(b-d)}{a+b}+a-\gamma\Big)
+\sqrt{\Big(\tfrac{\gamma(b-d)}{a+b}+a-\gamma\Big)^{2}-\tfrac{3\gamma d}{a+b}\Big(g_1-\tfrac{\gamma b}{a+b}-a+\gamma\Big)}\right]\,,
\]
so that $g^\star=\tilde g(\hat y)$, and Theorem~\ref{thm:saddlenode} yields a saddle-node at $g^\star$ when $g^\star>g_1$. The set of Figure~\ref{fig:twointerior} (drawn for $\delta=1$) yields the same qualitative picture under $\delta=2$, with $g^\star>g_1$ and two interior equilibria for $g_1<g<g^\star$ by Theorem~\ref{thm:classification}(iii). \\

\noindent\emph{Intermediate crowding $1<\delta<2$.}\quad 
Here $L$ is no longer concave: the first term of $L''$ in~\eqref{Lderivs} diverges as $y\to0^+$, 
so $L''(0^+)=+\infty$ and $L$ is convex near $y=0$. Neither case of Lemma~\ref{lem:concavity} holds, so Theorem~\ref{thm:classification} is replaced by Proposition~\ref{prop:nonconcave} of Subsection~\ref{subsec:nonconcave}, which permits as many as three interior equilibria, against at most two when $L$ is concave. By Theorem~\ref{thm:transcritC} the transcritical bifurcation at $C$ holds for every $\delta>0$ under~\eqref{ndC}, while Theorem~\ref{thm:transcritB} for $B$ requires $\delta\in\N$ and does not apply.

For the parameter set of Figure~\ref{fig:threeinterior}, which satisfies~\eqref{GA}, one finds
\[
g_0\approx0.0072<g_1\approx0.0215\,,
\]
and $\tilde g$~\eqref{tildeg} has an interior local minimum at $\check y\approx0.120$, of value $\tilde g(\check y)\approx0.0208$, besides its maximizer $\hat y\approx0.388$ with $g^\star=\tilde g(\hat y)\approx g_1$. At $g\approx0.0212$ the reduced equation $L(y)=0$ has three roots,
\[
y\approx0.026\,,\qquad y\approx0.260\,,\qquad y\approx0.486\,,
\]
hence three interior equilibria $(K(y_i),y_i)$, $i=1,2,3$: a saddle, a stable node, and a saddle (Proposition~\ref{prop:nonconcave}), as in Figure~\ref{fig:threeinterior}. For $g_0<g<g_1$ the boundary states $B$ and $C$ are also stable (Theorems~\ref{thm:stabB} and~\ref{thm:stabC}), so $B$, $C$, and the interior node give three stable states and the reef is tristable. These roots occur only for $g\in(\tilde g(\check y),g_1)$, a window of width about $6.6\times10^{-4}$, and require the interior minimum of $\tilde g$. When $\tilde g$ is monotone on $[0,y_C]$, the ordering $g_0<g_1$ of Lemma~\ref{lem:ordering} makes it decrease from $\tilde g(0)=g_1$ to $\tilde g(y_C)=g_0$, so $g^\star=g_1$ and at most one interior equilibrium remains. Figure~\ref{fig:numfps} places the $1<\delta<2$ regime within the full $(\delta,g)$ count.


\subsection{Numerical illustrations}
We illustrate the foregoing analysis by numerical simulation of system~\eqref{dxdy}. 

Figure~\ref{fig:phaseplane} shows the nullclines and phase portraits in the low-, intermediate-, and high-grazing regimes of Theorem~\ref{thm:classification}.

Figure~\ref{fig:bifdiagram} shows the dependence of the interior equilibrium and the coral-only state on the crowding exponent, for $\delta=2$ and $\delta=5$. 

Figure~\ref{fig:twointerior} shows the bifurcation diagram and phase portraits for a representative $\delta=1$ set with two interior equilibria.

Figure~\ref{fig:threeinterior} shows the bifurcation diagram and phase portrait for a $\delta=1.8$ set with three interior equilibria.

Figure~\ref{fig:numfps} maps the number of interior equilibria across the $(\delta,g)$ plane, gathering the regimes of Figures~\ref{fig:phaseplane}--\ref{fig:threeinterior}.

\begin{figure}[H]
\centering
\includegraphics[width=\textwidth]{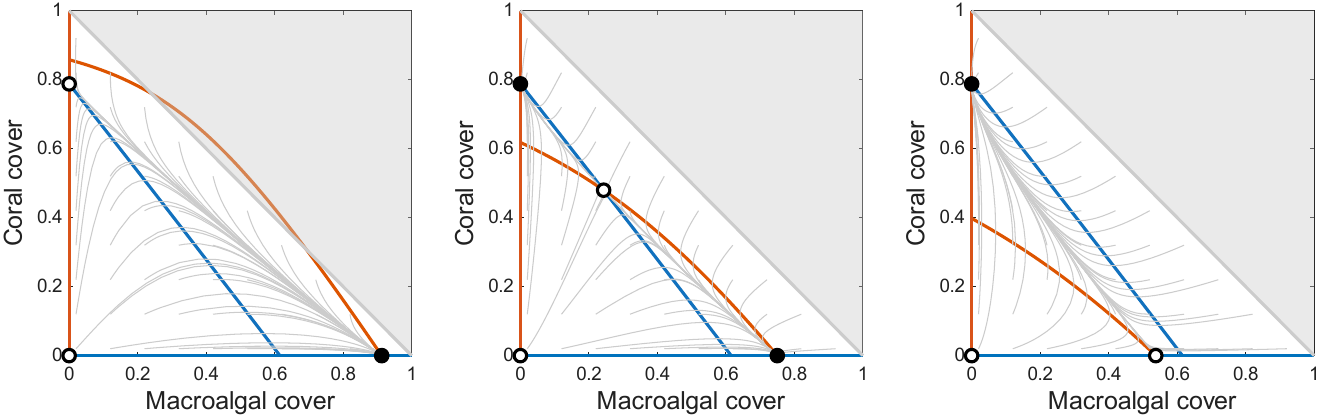}
\caption{Nullclines, equilibria, and stability as the grazing rate increases ($g=0.05$, $0.18$, and $0.35$ for the left, middle, and right panels, respectively). Orange curves are the $\dot x=0$ macroalgae nullclines, blue curves the $\dot y=0$ coral nullclines, and gray curves are sample trajectories. The gray triangle lies outside the domain. The crowding parameter is fixed at $\delta=2$, and the remaining parameters $(a,b,\gamma,m_0,d,n)=(0.3,\,1,\,0.8,\,0.2,\,0.02,\,0.02)$ lie within the ranges of Table~\ref{tab:params}. Filled and open dots mark stable and unstable equilibria, respectively. The coral-only state $C$ is stable for sufficiently high grazing, whereas the macroalgae-only state $B$ is stable only for sufficiently low grazing. For intermediate grazing (middle panel) the interior equilibrium exists and is a saddle, producing bistability, so that the resulting state is determined by the initial cover. The trivial state $O$ is unstable throughout. The three panels correspond to the low-, intermediate-, and high-grazing regimes of Theorem~\ref{thm:classification}.  }
\label{fig:phaseplane}
\end{figure}

\begin{figure}[H]
\centering
\includegraphics[width=\textwidth]{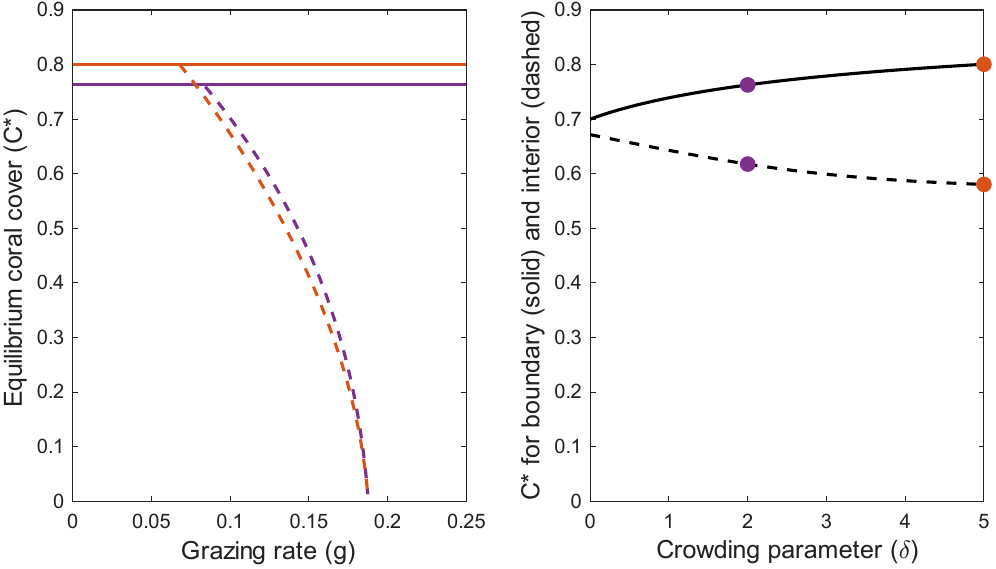}
\caption{The effects of crowding when there is one internal equilibrium. (Left) Bifurcation diagram of the equilibrium coral cover ($C^\ast$) versus the grazing rate $g$. Solid and dashed curves represent stable and unstable equilibria, respectively. The purple curves use $\delta=2$ and the red curves use $\delta=5$. The horizontal curves represent $y_C$ in the equilibrium $C$, which does not vary with $g$ but increases as $\delta$ increases. The interior equilibrium $E^\ast$ shifts to the left as $\delta$ increases. The remaining parameters are $(a,b,\gamma,m_0,d,n)=(0.3,\,1,\,0.6,\,0.15,\,0.15,\,0.02)$. (Right) Summary of the values of the equilibria as $\delta$ varies, with $\delta=2$ and $\delta=5$ indicated by dots matching the colors from the left figure. Here the grazing rate is fixed at $g=0.1$.} 
\label{fig:bifdiagram}
\end{figure}

\begin{figure}[H]
\centering
\includegraphics[width=\textwidth]{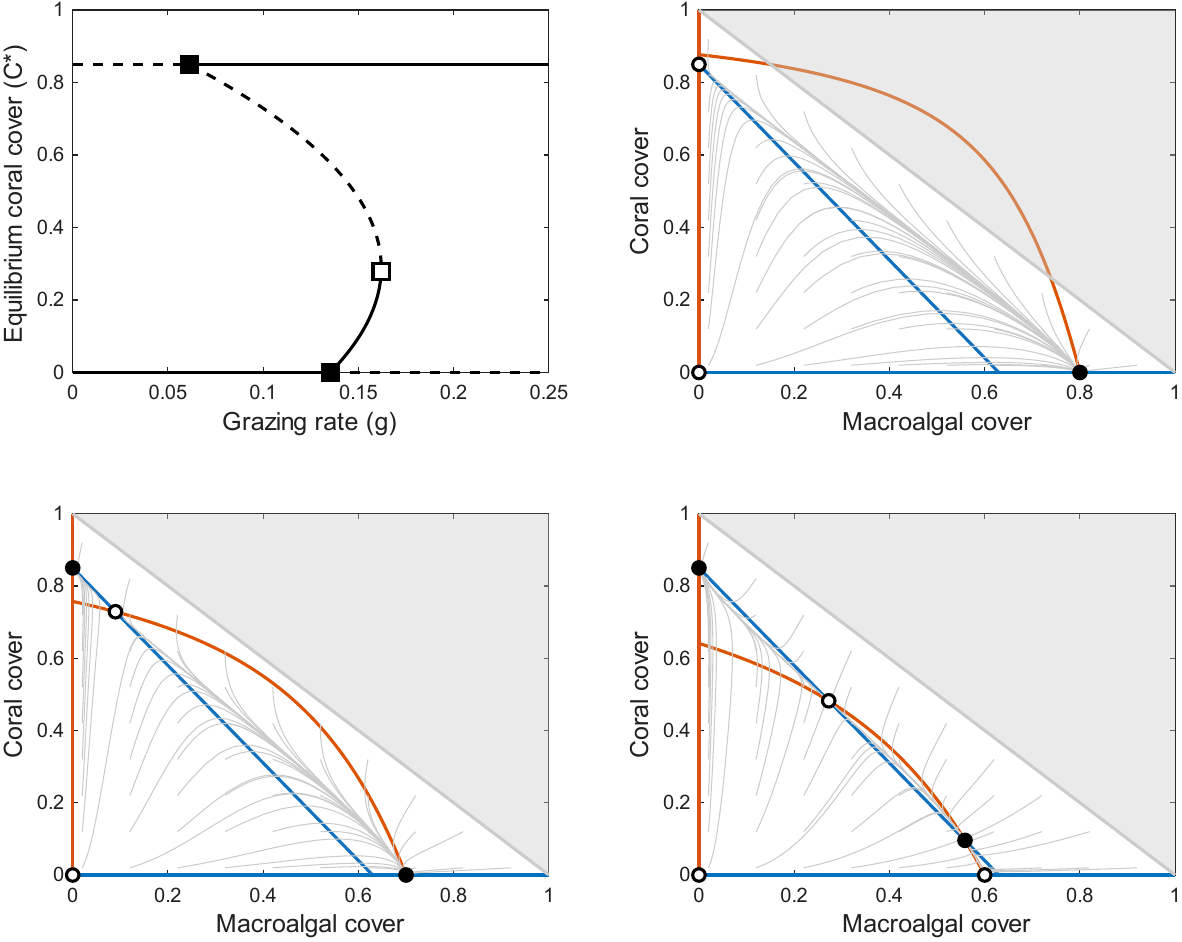}
\caption{Bifurcation and phase diagrams for the parameter set $(a,b,\gamma,m_0,d,n,\delta)=(0.45,\,0.9,\,0.5,\,0.05,\,0.1,\,0.05,\, 1)$ to summarize our findings when there are two interior equilibria. (Top left) Bifurcation diagram of the equilibrium coral cover ($C^*$) versus grazing rate ($g$). The solid and open squares indicate transcritical and saddle node bifurcations, respectively. The remaining panels are phase diagrams similar to those provided in Figure \ref{fig:phaseplane} but with this parameter set. The grazing rates used are $0.05$ (top right), $0.1$ (bottom left), and $0.15$ (bottom right).}
\label{fig:twointerior}
\end{figure}

\begin{figure}[H]
\centering
\includegraphics[width=\textwidth]{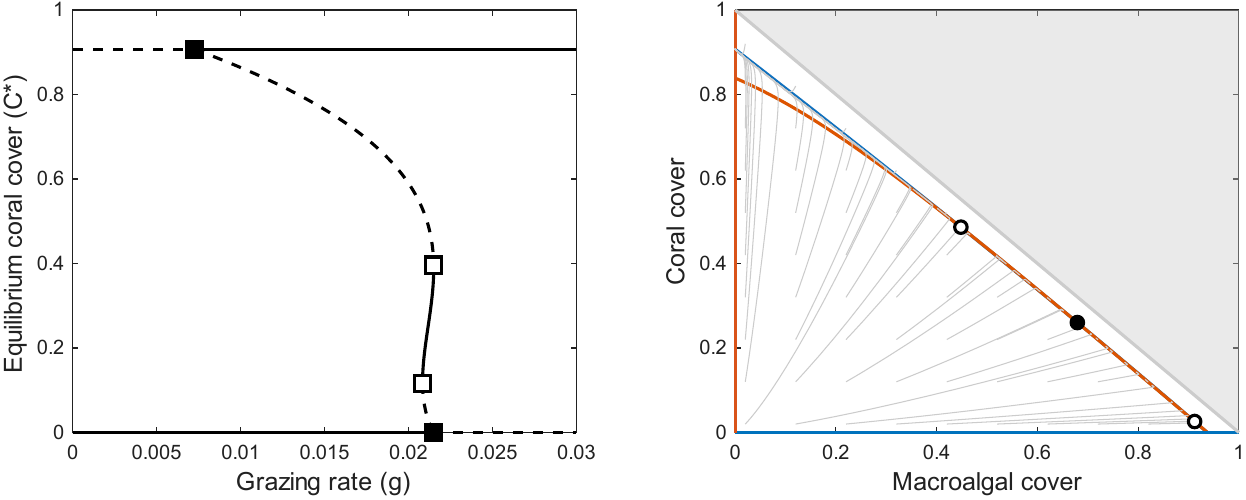}
\caption{Bifurcation and phase diagrams for the parameter set $(a,b,\gamma,m_0,d,n,\delta)=(0.035,\,1,\,0.82,\,0.03,\,0.075,\,0.03,\,1.8)$ to summarize our findings when there are three interior equilibria. (Left) Bifurcation diagram of the equilibrium coral cover ($C^*$) versus grazing rate ($g$). The solid and open squares indicate transcritical and saddle node bifurcations, respectively. The right panel is a phase diagram similar to those provided in Figure  \ref{fig:phaseplane} but with this parameter set and a grazing rate ($g$) of $0.0212$. The marked bifurcations are determined numerically for this non-integer $\delta$. }
\label{fig:threeinterior}
\end{figure}

\begin{figure}[H]
\centering
\hspace*{-1in}
\includegraphics[scale=0.75]{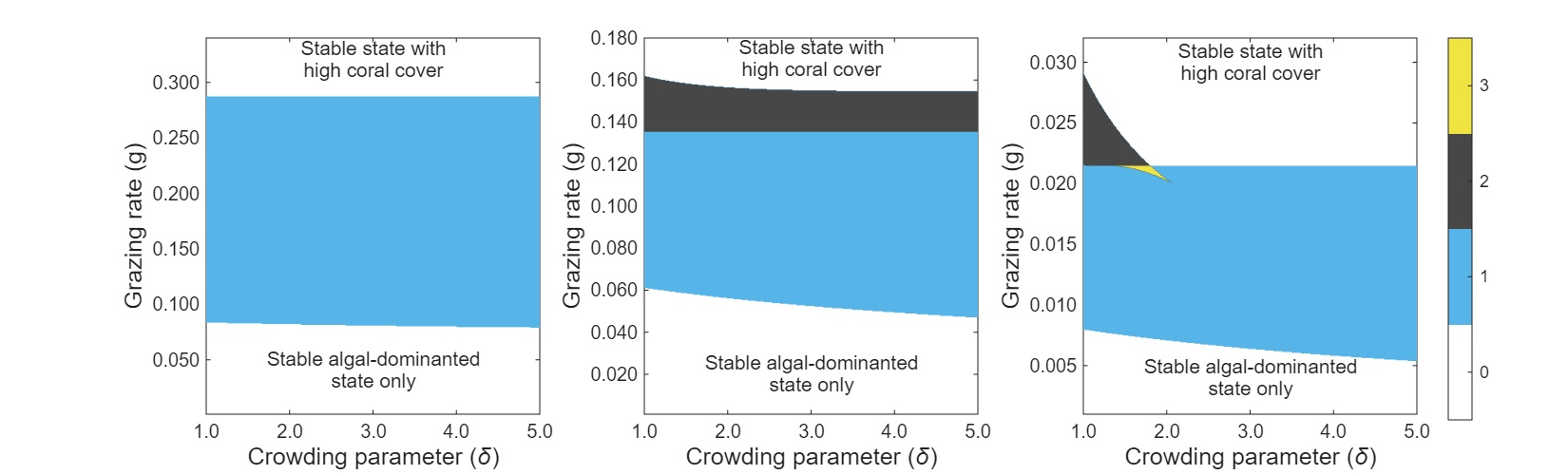}
\caption{Parameter space showing the number of interior fixed points as the crowding parameter ($\delta$) and grazing rate ($g$) are varied. The panels summarize the dynamics from Figures~\ref{fig:phaseplane}--\ref{fig:threeinterior} under different parameter regimes: left panel (parameters as in Figure~\ref{fig:phaseplane}), middle panel (Figure~\ref{fig:twointerior}), and right panel (Figure~\ref{fig:threeinterior}). Note that the $g$-axis covers a smaller range in the right panel for visual clarity.}
\label{fig:numfps}
\end{figure}

\bigskip

\section{Conclusion}\label{sec:conclusion}
We have analyzed a coral--macroalgae competition model in which coral mortality increases with cover through a crowding term $dC^{\delta}$. The system has no periodic orbit (Lemma~\ref{lem:dulac}), and its dynamics are organized by the grazing rate $g$ through two ordered thresholds $0<g_0<g_1<\gamma-n$: the coral-only state $C$ is stable for $g>g_0$ and the macroalgae-only state $B$ for $g<g_1$, and both are stable on the intermediate range. As $g$ crosses these thresholds, $C$ and $B$ exchange stability with a coexistence equilibrium through transcritical bifurcations (Theorems~\ref{thm:transcritC} and~\ref{thm:transcritB}, the latter for $\delta\in\N$). When the reduced function $L$ is concave, as in the constant-mortality model of~\cite{Mum07,Li14}, high grazing admits at most two coexistence states, a stable node and a saddle that collide in a saddle-node bifurcation at a third threshold $g^\star$ (Theorem~\ref{thm:saddlenode}). Intermediate crowding $1<\delta<2$ makes $L$ non-concave and can produce as many as three coexistence states (Proposition~\ref{prop:nonconcave}), so that the reef can be tristable, a regime beyond the reach of constant mortality. More broadly, the analysis shows how density-dependent crowding reshapes the coral--macroalgae dynamics across all crowding exponents: it shifts the grazing thresholds and the coral nullcline for every $\delta>0$, whereas this qualitatively new behavior is confined to that intermediate regime.

These results have direct ecological implications. The dynamics take place in the positively invariant region~$\Omega$, where the model has no limit cycle (Lemma~\ref{lem:dulac}). Coral and macroalgae cover therefore cannot settle into sustained oscillations.
We have characterized the local stability of the equilibria, and the phase portraits of Section~\ref{sec:exnum} show the reef approaching one of the stable ones according to the grazing regime and the initial cover. The bistable window $g_0<g<g_1$ is, moreover, the source of hysteresis: as grazing declines from a level at which coral dominates, the coral-dominated state remains stable until $g$ falls below $g_0$, where it loses stability through the transcritical bifurcation of Theorem~\ref{thm:transcritC} and the reef shifts abruptly to macroalgae dominance. Recovery does not retrace this path, since the macroalgae-dominated state remains stable throughout $g<g_1$: returning grazing to $g_0$ leaves the reef macroalgae-dominated, and coral is regained only once $g$ exceeds the higher threshold $g_1$, where the macroalgae-dominated state itself loses stability (Theorem~\ref{thm:transcritB}, for integer $\delta$). The gap $g_1-g_0$ thus measures the additional grazing that recovery demands, and the wider it is, the more resistant a degraded reef is to reversal.

Beyond the two dominance states, the model also admits a mixed coral--macroalgae community. In the bistable regime this interior equilibrium is a saddle, separating the two outcomes rather than offering a state the reef can rest in. Under stronger grazing it is joined by a stable node, an observable mixed community attained for a range of initial cover, and the two collide and disappear at $g^\star$ in the saddle-node bifurcation of Theorem~\ref{thm:saddlenode}, a further tipping point beyond which coral dominance is the only stable state. Crowding adds one further possibility: for intermediate $1<\delta<2$ the mixed community can be stable alongside both dominance states, so that the reef admits three stable communities, the one realized being selected by the initial cover.

\bigskip

\begin{appendix}
\section{Proof of Proposition~\ref{prop:sn}}\label{app:proofs}
To prove Proposition~\ref{prop:sn} it remains to carry out, for each of the three nonhyperbolic equilibria, the normal-form reduction of \cite[Thm.~1, Sec.~2.11]{Pe01} (equivalently \cite[Thm.~7.1, Chap.~2]{Zha92}): translate the equilibrium to the origin, reduce to the center manifold, and read off the order $m$ of the leading term of the reduced flow, with $m=2$ a saddle-node. Throughout, a dot denotes differentiation with respect to the original time and a subscript $\tau$ with respect to the rescaled time.\\

\noindent\textbf{\textup{(i)} The macroalgae-only state $B$ at $g=g_1$.}\quad
At $g=g_1$ the equilibrium $B$ is nonhyperbolic, with eigenvalues $\lambda_2=0$ and $\lambda_1<0$.
Translating $B$ to the origin by $(X,Y)=(x-x_B,y)$, with $x_B$ as in~\eqref{xB}, and expanding in a power series to fourth order, we obtain
\begin{subequations}\label{eq:B Trafo 1}
\begin{align}
\dot X &= \lambda_1X+x_B\left(a-\gamma - g_1\right)Y+\left(a-\gamma - g_1\right)XY-\gamma X^2- g_1x_BY^2\\&\quad - g_1 XY^2- g_1x_BY^3-g_1XY^3- g_1x_BY^4+P_1(X,Y)\,, \notag\\
\dot Y &= -(a+b) XY-bY^2 -dY^{\delta +1}.
\end{align}
\end{subequations}
where $P_1(X,Y)$ is the power series collecting the terms $X^iY^j$ with $i+j\geq5$. The linear change of variables 
\begin{align*}
\begin{pmatrix}
X\\
Y
\end{pmatrix}=
\begin{pmatrix}
\frac{-x_B(a-\gamma-g_1)}{\lambda_1}& \frac{1}{\lambda_1}\\
1 & 0
\end{pmatrix}
\begin{pmatrix}
\bar{x}\\
\bar{y}
\end{pmatrix}
\end{align*}
brings the linear part of~\eqref{eq:B Trafo 1} to diagonal form, turning it into
\begin{equation*}
\begin{split}
\dot{\bar{x}} &=  a_{1,1}  \bar{x}\bar{y}+  
a_{2,0} \bar{x}^2+a_{\delta+1,0}\bar{x}^{\delta+1}
\\
\dot{\bar{y}} &= b_{1,0}\bar{y}+b_{1,1} \bar{x}\bar{y}+b_{0,2}\bar{y}^2+b_{2,0}\bar{x}^2+b_{2,1} \bar{x}^2\bar{y}+b_{3,0}\bar{x}^3+b_{\delta+1,0}\bar{x}^{\delta+1}+P_2(\bar{x},\bar{y}),
\end{split}
\end{equation*}
with
\begin{equation*}
\begin{split}
a_{1,1}&:= -\frac{a+b}{\lambda_1}\,,\qquad a_{2,0}:= \frac{x_B(a-\gamma-g_1)(a+b)}{\lambda_1}-b\,,\qquad a_{\delta+1,0}:= -d\,, \\
b_{1,0}&:= \lambda_1\,, \qquad b_{0,2}:=-\frac{\gamma}{\lambda_1}\,,\qquad b_{2,1}:= -g_1\,,\\
b_{1,1}&:= a-\gamma-g_1+\frac{x_B(a-\gamma-g_1)(2\gamma-(a+b))}{\lambda_1}\,,\\
b_{2,0}&:= \frac{x_B^2(a-\gamma-g_1)^2(a+b-\gamma)}{\lambda_1}-x_B(a-\gamma-g_1)(a+b-\gamma-g_1)-\lambda_1g_1x_B\,,\\
b_{3,0}&:= g_1\left(x_B(a-\gamma-g_1)-\lambda_1x_B \right)\,, \qquad b_{\delta+1,0}:= -dx_B(a-\gamma-g_1)\,,
\end{split}
\end{equation*}
and $P_2(\bar{x},\bar{y})$ a power series of order $\geq4$. Rescaling time by $\tau:=\lambda_1 t$ then gives 
\begin{subequations}\label{eq:B Trafo 2}
\begin{align}
\bar{x}_\tau &= c_{1,1} \bar{x}\bar{y}+
c_{2,0}\bar{x}^2+c_{\delta+1,0}\bar{x}^{\delta+1}\,,\label{eq:B Trafo 2a}
\\
\bar{y}_\tau &= \bar{y}+d_{1,1}\bar{x}\bar{y}+d_{0,2}\bar{y}^2+
d_{2,0}\bar{x}^2+d_{2,1}\bar{x}^2\bar{y}+d_{3,0}\bar{x}^3+d_{\delta+1,0}\bar{x}^{\delta+1}+P_3(\bar{x},\bar{y})\,.
\end{align}
\end{subequations}
where $c_{i,j}:=\lambda_1^{-1}a_{i,j}$ and $d_{i,j}:=\lambda_1^{-1}b_{i,j}$, and $P_3(\bar{x},\bar{y})$ is a power series of order $\geq4$. For $\delta\in\N$ the system is thus in the normal form
\begin{equation}\label{eq:transformed_vf_B}
\bar{x}_\tau=P(\bar{x},\bar{y})\,, \qquad  \bar{y}_\tau=\bar{y}+Q(\bar{x},\bar{y})\,,
\end{equation}
of \cite[Thm.~1, Sec.~2.11]{Pe01}, with $P,Q$ analytic near the origin of order $\geq2$. It remains to reduce to the center manifold. Set $\mathcal{H}(\bar{x},\bar{y}):=\bar{y}+Q(\bar{x},\bar{y})$. Since $Q$ carries no constant or linear term, $\mathcal{H}(0,0)=0$ and $\mathcal{H}_{\bar{y}}(0,0)=1\neq0$, so by the Implicit Function Theorem $\bar{y}+Q(\bar{x},\bar{y})=0$ has near the origin a unique analytic solution $\bar{y}=\phi(\bar{x})$ with $\phi(\bar{x})+Q(\bar{x},\phi(\bar{x}))\equiv0$ and $\phi(0)=\phi'(0)=0$. Writing $\phi(\bar{x})=A_2\bar{x}^2+A_3\bar{x}^3+\dotsc$ and matching powers of $\bar{x}$ in $\mathcal{H}(\bar{x},\phi(\bar{x}))=0$ yields
\begin{equation*}
\begin{split}
0&=\mathcal{H}(\bar{x},\phi(\bar{x}))=\phi(\bar{x})+Q(\bar{x},\phi(\bar{x})) \\
&= (A_2 + d_{2,0} + d_{\delta+1,0}\mathds{1}_{\{ \delta =1\}} )  \bar{x}^2 + (A_3 + d_{1,1}A_2 + d_{3,0} + d_{\delta+1,0}\mathds{1}_{\{ \delta =2\}})  \bar{x}^3 + O(\bar{x}^4)\,,
\end{split}
\end{equation*}
as $\bar{x} \to 0$, where $\mathds{1}_{\{ \delta =k\}}=1$ if $\delta=k$ and $0$ otherwise, whence $A_2=-(d_{2,0}+d_{\delta+1,0}\mathds{1}_{\{\delta=1\}})$ and
\begin{equation*}
\begin{split}
A_3&=-(d_{1,1} A_2 + d_{3,0} + d_{\delta+1,0}\mathds{1}_{\{ \delta =2\}})\\
&=d_{1,1}(d_{2,0} + d_{\delta+1,0}\mathds{1}_{\{ \delta =1\}}) - d_{3,0} - d_{\delta+1,0}\mathds{1}_{\{ \delta =2\}}\,.
\end{split}
\end{equation*}
The reduced flow on the center manifold is given by $\psi(\bar{x}):= P(\bar{x},\phi(\bar{x}))$, where $P(\bar{x},\bar{y})=c_{1,1}\bar{x}\bar{y}+c_{2,0}\bar{x}^2+c_{\delta+1,0}\bar{x}^{\delta+1}$ is the right-hand side of~\eqref{eq:B Trafo 2a}. We obtain
\begin{equation*}
\psi(\bar{x}) = (c_{2,0} + c_{\delta+1,0}\mathds{1}_{\{ \delta =1\}}) \bar{x}^2 + ( c_{1,1}A_2 +c_{\delta+1,0}\mathds{1}_{\{ \delta =2\}}) \bar{x}^3 + O(\bar{x}^4)
\end{equation*}
as $\bar{x} \to 0$. The nondegeneracy condition~\eqref{ndB} is equivalent, by $\lambda_1\neq 0$ and the identity
\[
\frac{x_B(a-\gamma-g_1)(a+b)}{\lambda_1}=-\frac{(a-\gamma-g_1)(a+b)}{\gamma}\,,
\]
to $c_{2,0} + c_{\delta+1,0}\mathds{1}_{\{ \delta =1\}}\neq 0$, so that the leading term of $\psi$ is quadratic. Hence $m = 2$, and $B$ is a saddle-node by \cite[Thm.~1, Sec.~2.11]{Pe01}.\\
    
\noindent\textbf{\textup{(ii)} The coral-only state $C$ at $g=g_0$.}\quad
At $g=g_0$ the equilibrium $C$ is nonhyperbolic, with eigenvalues $\mu_1=0$ and $\mu_2<0$.
Translating $C$ to the origin by $(X,Y)=(x,y-y_C)$ and expanding in a power series to fourth order, we obtain
\begin{subequations}\label{eq:C Trafo 1}
\begin{align}
\dot X&= -\gamma X^2 + \bigg((a-\gamma) - \frac{g_0}{(1-y_C)^2} \bigg) XY
- \frac{g_0}{(1-y_C)^3} XY^2 - \frac{g_0}{(1-y_C)^4} XY^3\\
&\quad + P_1(X,Y)\,,\notag\\
\dot Y &= -(a+b)y_C X + \mu_2 Y
-(a+b)XY - \bigg(b+ \frac{1}{2} d \delta (\delta +1) y_C^{\delta -1}\bigg) Y^2\\
& \quad
- \frac{1}{6} d (\delta -1) \delta (\delta +1) y_C^{\delta -2} Y^3
- \frac{1}{24} d (\delta -2)(\delta -1) \delta (\delta +1) y_C^{\delta -3} Y^4 \notag \\
&\quad + Q_1(X,Y)\,, \notag 
\end{align}
\end{subequations}
where $P_1(X,Y)$ and $Q_1(X,Y)$ are power series collecting the terms $X^iY^j$ with $i+j\geq5$. The change of variables $\bar{x}=X$, $\bar{y}=-(a+b)y_C X + \mu_2 Y$, $\tau=\mu_2 t$ brings~\eqref{eq:C Trafo 1} into the normal form of \cite[Eq.~(2), Sec.~2.11]{Pe01},
\begin{equation}\label{eq:Zhang form}
\bar{x}_\tau =P_2(\bar{x}, \bar{y}), \qquad \bar{y}_\tau =\bar{y} + Q_2(\bar{x}, \bar{y})\,,
\end{equation}
with $P_2$ and $Q_2$ analytic near the origin of order $\geq2$. Carrying out the substitution and expanding gives
\begin{subequations}\label{eq:saddle node C}
\begin{equation}\label{eq:saddle node C_1}
\bar{x}_\tau = P_2(\bar{x}, \bar{y}) = A_{2,0} \bar{x}^2 +  A_{1,1} \bar{x}\bar{y} + A_{3,0} \bar{x}^3 + A_{2,1} \bar{x}^2\bar{y} +  A_{1,2} \bar{x}\bar{y}^2 +  \dotsc\,,
\end{equation}
where
\begin{equation*}
\begin{split}
A_{2,0}&:=  \frac{(a+b)y_C}{\mu_2^2} \bigg( (a-\gamma) - \frac{g_0}{(1-y_C)^2}\bigg)-\frac{\gamma}{\mu_2}\,, \qquad  A_{1,1}:= \frac{1}{\mu_2^2} \bigg( (a-\gamma) - \frac{g_0}{(1-y_C)^2}\bigg)\,, \\
A_{3,0}&:=  - \frac{(a+b)^2 y_C^2}{\mu_2^3} \frac{g_0}{(1-y_C)^3}\,, \qquad A_{2,1}:= - \frac{2(a+b) y_C}{\mu_2^3} \frac{g_0}{(1-y_C)^3}\,,\\
A_{1,2}&:= - \frac{1}{\mu_2^3} \frac{g_0}{(1-y_C)^3}\,, \, \dotsc \, ,
\end{split}
\end{equation*}
and
\begin{align}\label{eq:saddle node C_2}
\bar{y}_\tau &= \bar{y} + Q_2(\bar{x}, \bar{y}) \\
&= \bar{y} + B_{1,1} \bar{x}\bar{y}
+ B_{2,0} \bar{x}^2 + B_{0,2} \bar{y}^2 +  B_{3,0} \bar{x}^3
+ B_{0,3} \bar{y}^3 + B_{1,2} \bar{x}\bar{y}^2 +  B_{2,1} \bar{x}^2\bar{y}  + \dotsc\, ,  \notag
\end{align}
where
\begin{equation*}
\begin{split}
B_{1,1}&:= - \frac{(a+b)}{\mu_2} - \frac{(a+b)y_C}{\mu_2^2}
\bigg( (a-\gamma) - \frac{g_0}{(1-y_C)^2} + 2b + d \delta  (\delta +1) y_C^{\delta-1} \bigg)\,,
\\
B_{2,0} &:= \frac{(a+b)y_C}{\mu_2} \big(\gamma - (a+b)\big)
- \frac{(a+b)^2y_C^2}{\mu_2^2} \bigg( (a-\gamma) - \frac{g_0}{(1-y_C)^2} + b + \frac{1}{2} d \delta (\delta +1) y_C^{\delta-1} \bigg)\,,
\\
B_{0,2} &:= - \frac{d\delta (\delta+1)y_C^{\delta -1} + 2b}{2\mu_2^2}\,,
\\
B_{3,0}  &:=  \frac{(a+b)^3y_C^3}{\mu_2^3} \bigg( \frac{g_0}{(1-y_C)^3} - \frac{d(\delta-1)\delta (\delta + 1)y_C^{\delta-2}}{6} \bigg)\,,
\\
B_{0,3}  &:=  - \frac{d(\delta-1)\delta (\delta + 1)y_C^{\delta-2}}{6\mu_2^3}\,,
\\
B_{1,2}  &:=  \frac{(a+b)y_C}{\mu_2^3} \bigg( \frac{g_0}{(1-y_C)^3} - \frac{d(\delta-1)\delta (\delta + 1)y_C^{\delta-2}}{2}\bigg)\, ,
\\
B_{2,1}  &:=  \frac{2(a+b)^2y_C^2}{\mu_2^3} \bigg( \frac{g_0}{(1-y_C)^3} - \frac{d(\delta-1)\delta (\delta + 1)y_C^{\delta-2}}{4} \bigg)\,,    \dotsc\,.
\end{split}
\end{equation*}
\end{subequations}
By the Implicit Function Theorem there is, near $\bar{x}=0$, a unique analytic $\phi$ with $\phi(\bar{x})+Q_2(\bar{x},\phi(\bar{x}))\equiv0$, $\phi(0)=0$, and $\phi'(0)=0$, expanding as $\phi(\bar{x})=c_2\bar{x}^2+c_3\bar{x}^3+\dotsc$. Since $Q_2(\bar{x},\bar{y})=B_{1,1}\bar{x}\bar{y}+B_{2,0}\bar{x}^2+B_{0,2}\bar{y}^2$ up to terms of degree $\geq3$, inserting $\bar{y}=\phi(\bar{x})=c_2\bar{x}^2+O(\bar{x}^3)$ into $\bar{y}+Q_2(\bar{x},\bar{y})=0$ gives
\begin{equation*}
0 = \phi(\bar{x}) + Q_2 (\bar{x}, \phi(\bar{x})) = (c_2 +B_{2,0}) \bar{x}^2 + O(\bar{x}^3)\,,
\end{equation*}
whence $c_2 = - B_{2,0}$ and $\phi(\bar{x}) =- B_{2,0} \bar{x}^2 + O(\bar{x}^3)$. The reduced flow $\psi(\bar{x}):= P_2 (\bar{x}, \phi(\bar{x}))$ then satisfies
\begin{equation*}
\psi(\bar{x}) = A_{2,0} \bar{x}^2 + A_{1,1}  \bar{x} \phi(\bar{x}) + A_{3,0}  \bar{x}^3 + A_{2,1} \bar{x}^2 \phi(\bar{x}) + A_{1,2} \bar{x} \phi(\bar{x})^2 + \dotsc\,,
\end{equation*}
so that its leading term is $A_{2,0} \bar{x}^2$. The nondegeneracy condition~\eqref{ndC} is equivalent, after rearrangement, to $A_{2,0}\neq 0$, so that $m=2$. Hence $C$ is a saddle-node by \cite[Thm.~1, Sec.~2.11]{Pe01}.\\

\noindent\textbf{\textup{(iii)} The interior equilibrium $E_3^\ast$ at $g=g^\ast$.}\quad
At $g=g^\star$ the interior equilibrium $E_3^\star=(x_3^\star,y_3^\star)$ is nonhyperbolic, with eigenvalues $0$ and $\nu_2:=\operatorname{tr}J(E_3^\star)<0$ and with $L'(y_3^\star)=0$ (Theorem~\ref{thm:classification}). The right-hand side of~\eqref{dxdy} is real-analytic near $E_3^\star$ for every $\delta>0$. Translating $E_3^\star$ to the origin by $(X,Y)=(x-x_3^\star,y-y_3^\star)$ and expanding in a power series to fourth order, we obtain
\begin{subequations}\label{eq:E3_Trafo_1}
\begin{align}
\dot X &= -x_3^\star\gamma X + x_3^\star \partial_y F(x_3^\star,y_3^\star) Y \\
&\quad -\gamma X^2 + \partial_y F(x_3^\star,y_3^\star) XY - \frac{gx_3^\star}{(1-y_3^\star)^3}Y^2 - \frac{g}{(1-y_3^\star)^3}XY^2 \notag\\
&\quad - \frac{gx_3^\star}{(1-y_3^\star)^4}Y^3 - \frac{g}{(1-y_3^\star)^4}XY^3 - \frac{gx_3^\star}{(1-y_3^\star)^5}Y^4 + P_1(X,Y)\, , \notag\\
\dot Y &= -(a+b)y_3^\star X - y_3^\star\big(b+\delta d(y_3^\star)^{\delta-1}\big)Y \\
&\quad -(a+b)XY - \Big(b+\tfrac{1}{2}\delta(\delta+1)d(y_3^\star)^{\delta-1}\Big)Y^2 \notag\\
&\quad - \tfrac{1}{6}\delta(\delta-1)(\delta+1)d(y_3^\star)^{\delta-2}Y^3 \notag\\
&\quad - \tfrac{1}{24}\delta(\delta-2)(\delta-1)(\delta+1)d(y_3^\star)^{\delta-3}Y^4 + Q_1(X,Y)\,. \notag
\end{align}
\end{subequations}
The functions $P_1(X,Y)$ and $Q_1(X,Y)$ are power series collecting the terms $X^iY^j$ with $i+j\geq 5$, the linear part is the Jacobian $J(E_3^\star)$ of~\eqref{eq:Jacobi_iE}, and $\partial_y F(x_3^\star,y_3^\star)=(a-\gamma)-\tfrac{g}{(1-y_3^\star)^2}$, $\partial_y G(x_3^\star,y_3^\star)=-(b+\delta d(y_3^\star)^{\delta-1})$. The change of variables
\begin{equation}\label{eq:E3_Trafo_2}
\bar{x} = X - \frac{x_3^\star\gamma}{(a+b)y_3^\star} Y\, , \qquad \bar{y} = -X - \frac{b+\delta d(y_3^\star)^{\delta-1}}{a+b}Y\, , \qquad \tau = \nu_2 t\, ,
\end{equation}
with $\nu_2=-x_3^\star\gamma-y_3^\star\big(b+\delta d(y_3^\star)^{\delta-1}\big)$, has $Y$-coefficients fixed by $L'(y_3^\star)=0$. Rearranging this equality via~\eqref{Lstar} gives $\partial_y F(x_3^\star, y_3^\star)=-\tfrac{\gamma(b+\delta d(y_3^\star)^{\delta-1})}{a+b}$, hence $\tfrac{\partial_y F(x_3^\star, y_3^\star)}{\partial_y G(x_3^\star, y_3^\star)}=\tfrac{\gamma}{a+b}$ on dividing by $\partial_y G(x_3^\star, y_3^\star)=-(b+\delta d(y_3^\star)^{\delta-1})<0$, so that the constants in~\eqref{eq:E3_Trafo_2} are $\kappa:=\tfrac{x_3^\star\gamma}{(a+b)y_3^\star}$ and $\sigma:=\tfrac{1}{\gamma}\partial_y F(x_3^\star,y_3^\star)=-\tfrac{b+\delta d(y_3^\star)^{\delta-1}}{a+b}$. The change of variables~\eqref{eq:E3_Trafo_2} brings the system into the normal form $\bar{x}_\tau=P_2(\bar{x},\bar{y})$, $\bar{y}_\tau=\bar{y}+Q_2(\bar{x},\bar{y})$ of \cite[Eq.~(2), Sec.~2.11]{Pe01}, with $P_2,Q_2$ analytic near the origin of order $\geq2$. With $\Delta:=\sigma-\kappa=\tfrac{\nu_2}{(a+b)y_3^\star}$, carrying out the substitution and expanding give
\begin{subequations}\label{eq:saddle node E3}
\begin{equation}\label{eq:saddle node E3_1}
\bar{x}_\tau = P_2(\bar{x}, \bar{y}) = \alpha_{2,0}\bar{x}^2 + \alpha_{1,1}\bar{x}\bar{y} + \alpha_{0,2}\bar{y}^2 + O\big(|(\bar{x}, \bar{y})|^3\big)\,,
\end{equation}
where
\begin{equation*}
\begin{split}
\alpha_{2,0} &:= 
\frac{1}{\nu_2\Delta^2}\bigg[
-\frac{gx_3^\star}{(1-y_3^\star)^3} 
+ \kappa\Big((a+b)\sigma 
+ b + \tfrac{1}{2}\delta(\delta+1)
d\,(y_3^\star)^{\delta-1}\Big)\bigg]\,, 
\\
\alpha_{1,1} &:= 
\frac{1}{\nu_2\Delta^2}\bigg[
\Delta  \partial_y F(x_3^\star, y_3^\star)
- \frac{2gx_3^\star}{(1-y_3^\star)^3} 
 + \kappa\Big((a+b)(\sigma+\kappa)
+ 2b + \delta(\delta+1)
d(y_3^\star)^{\delta-1}\Big)
\bigg]\,, 
\\
\alpha_{0,2} &:= 
\frac{1}{\nu_2\Delta^2}\bigg[
-\gamma \kappa^2 + \partial_y F(x_3^\star, y_3^\star) \kappa 
 - \frac{gx_3^\star}{(1-y_3^\star)^3} 
+ \kappa\Big((a+b)\kappa
+ b + \tfrac{1}{2}\delta(\delta+1)
d(y_3^\star)^{\delta-1}\Big)
\bigg]\,. 
\end{split}
\end{equation*}
Moreover,
\begin{equation}\label{eq:saddle node E3_2}
\bar{y}_\tau = \bar{y} + Q_2(\bar{x}, \bar{y}) = \bar{y} + \beta_{2,0}\bar{x}^2 + \beta_{1,1}\bar{x}\bar{y} + \beta_{0,2}\bar{y}^2 + O\big(|(\bar{x}, \bar{y})|^3\big)\,,
\end{equation}
where
\begin{equation*}
\begin{split}
\beta_{2,0} &:= 
\frac{1}{\nu_2\Delta^2}\bigg[
\frac{gx_3^\star}{(1-y_3^\star)^3} 
- \sigma\Big((a+b)\sigma 
+ b + \tfrac{1}{2}\delta(\delta+1)
d\,(y_3^\star)^{\delta-1}\Big)
\bigg]\, , 
\\
\beta_{1,1} &:= 
\frac{1}{\nu_2\Delta^2}\bigg[
-\Delta  \partial_y F(x_3^\star, y_3^\star) 
+ \frac{2gx_3^\star}{(1-y_3^\star)^3} 
- \sigma\Big((a+b)(\sigma+\kappa)
+ 2b + \delta(\delta+1)
d(y_3^\star)^{\delta-1}\Big)
\bigg]\,, 
\\
\beta_{0,2} &:= 
\frac{1}{\nu_2\Delta^2}\bigg[
\gamma \kappa^2 - \partial_y F(x_3^\star, y_3^\star) \kappa  
+ \frac{gx_3^\star}{(1-y_3^\star)^3} 
 - \sigma\Big((a+b)\kappa
+ b + \tfrac{1}{2}\delta(\delta+1)
d(y_3^\star)^{\delta-1}\Big)
\bigg]\,.
\end{split}
\end{equation*}
\end{subequations}
Set $\mathcal{H}(\bar{x},\bar{y}):=\bar{y}+Q_2(\bar{x},\bar{y})$. Since $Q_2$ has no constant or linear part, $\mathcal{H}(0,0)=0$ and $\mathcal{H}_{\bar{y}}(0,0)=1\neq0$, so by the Implicit Function Theorem $\bar{y}+Q_2(\bar{x},\bar{y})=0$ has near the origin a unique analytic solution $\bar{y}=\phi(\bar{x})$ with $\phi(\bar{x})+Q_2(\bar{x},\phi(\bar{x}))\equiv0$ and $\phi(0)=\phi'(0)=0$. Writing $\phi(\bar{x})=c_2\bar{x}^2+c_3\bar{x}^3+\dotsc$ and matching the $\bar{x}^2$-coefficients in $\mathcal{H}(\bar{x},\phi(\bar{x}))=0$ gives
\begin{equation*}
0 = (c_2 + \beta_{2,0})\bar{x}^2 + O(\bar{x}^3),
\end{equation*}
hence $c_2 = -\beta_{2,0}$ and $\phi(\bar{x})=-\beta_{2,0}\bar{x}^2+O(\bar{x}^3)$. The reduced flow is given by $\psi(\bar{x}):= P_2(\bar{x},\phi(\bar{x}))$, for which
\begin{equation*}
\psi(\bar{x}) = \alpha_{2,0}\bar{x}^2 + \alpha_{1,1}\bar{x}\phi(\bar{x}) + \alpha_{0,2}\phi(\bar{x})^2 + \dotsc\, ,
\end{equation*}
so that the leading term of $\psi$ is $\alpha_{2,0}\bar{x}^2$. It remains to check that $\alpha_{2,0}\neq 0$, which gives $m=2$. Substituting $\kappa$ and $\sigma$ into $\alpha_{2,0}$ and using~\eqref{Lderivs} yields
\begin{equation*}
\begin{split}
\alpha_{2,0} 
= \frac{x_3^\star}{2\nu_2\Delta^2}\Big[\frac{\gamma \delta(\delta-1)d(y_3^\star)^{\delta-2}}{a+b}
-\frac{2g}{(1-y_3^\star)^3} 
\Big]
=  \frac{x_3^\star}{2\nu_2\Delta^2}  L''(y_3^\star) 
= \frac{(a+b)^2  x_3^\star (y_3^\star)^2}{2\nu_2^3} L''(y_3^\star)\, .
\end{split}
\end{equation*}
Since $L$ is strictly concave on $(0,y_C]$ and $x_3^\star,y_3^\star>0$ while $\nu_2^3<0$, we have $\alpha_{2,0}>0$. Hence $m=2$, and $E_3^\star$ is a saddle-node by \cite[Thm.~1, Sec.~2.11]{Pe01}.
\end{appendix}

\bigskip

\section*{Data Availability}
Data sharing is not applicable to this article as no datasets were generated or analyzed during the current study.




\begin{thebibliography}{99}
	
	\bibitem{Ban13}
	{\sc S.~S.~Ban, N.~A.~J.~Graham, and S.~R.~Connolly}, {\em Evidence for multiple stressor interactions and effects on coral reefs}, Glob. Change Biol., 20 (2014), pp.~681--697.
	
	\bibitem{Bayraktarov19}
	{\sc E.~Bayraktarov, P.~J.~Stewart-Sinclair, S.~Brisbane, L.~Bostr\"om-Einarsson, M.~I.~Saunders, C.~E.~Lovelock, H.~P.~Possingham, P.~J.~Mumby, and K.~A.~Wilson}, {\em Motivations, success, and cost of coral reef restoration}, Restor. Ecol., 27 (2019), pp.~981--991.
	
	\bibitem{Bell04}
	{\sc D.~R.~Bellwood, T.~P.~Hughes, C.~Folke, and M.~Nystr\"om}, {\em Confronting the coral reef crisis}, Nature, 429 (2004), pp.~827--833.
	
	\bibitem{Black12}
	{\sc J.~C.~Blackwood, A.~Hastings, and P.~J.~Mumby}, {\em The effect of fishing on hysteresis in Caribbean coral reefs}, Theor. Ecol., 5 (2012), pp.~105--114.
	
	\bibitem{Fun11}
	{\sc T.~Fung, R.~M.~Seymour, and C.~R.~Johnson}, {\em Alternative stable states and phase shifts in coral reefs under anthropogenic stress}, Ecology, 92 (2011), pp.~967--982.
	
	\bibitem{Harborne17}
	{\sc A.~R.~Harborne, A.~Rogers, Y.~M.~Bozec, and P.~J.~Mumby}, {\em Multiple stressors and the functioning of coral reefs}, Annu. Rev. Mar. Sci., 9 (2017), pp.~445--468.
	
	\bibitem{Hughes84}
	{\sc T.~P.~Hughes}, {\em Population dynamics based on individual size rather than age: a general model with a reef coral example}, Am. Nat., 123 (1984), pp.~778--795.
	
	\bibitem{Hugh94}
	{\sc T.~P.~Hughes}, {\em Catastrophes, phase shifts, and large-scale degradation of a Caribbean coral reef}, Science, 265 (1994), pp.~1547--1551.
	
	\bibitem{Hughes99}
	{\sc T.~P.~Hughes and J.~H.~Connell}, {\em Multiple stressors on coral reefs: a long-term perspective}, Limnol. Oceanogr., 44 (1999), pp.~932--940.
	
	\bibitem{Know92}
	{\sc N.~Knowlton}, {\em Thresholds and multiple stable states in coral reef community dynamics}, Am. Zool., 32 (1992), pp.~674--682.
	
	\bibitem{Lessios84}
	{\sc H.~A.~Lessios, D.~R.~Robertson, and J.~D.~Cubit}, {\em Spread of \emph{Diadema} mass mortality through the Caribbean}, Science, 226 (1984), pp.~335--337.
	
	\bibitem{Li14}
	{\sc X.~Li, H.~Wang, Z.~Zhang, and A.~Hastings}, {\em Mathematical analysis of coral reef models}, J. Math. Anal. Appl., 416 (2014), pp.~352--373.
	
	\bibitem{MoFo99}
	{\sc F.~Moberg and C.~Folke}, {\em Ecological goods and services of coral reef ecosystems}, Ecol. Econ., 29 (1999), pp.~215--233.
	
	\bibitem{Mum07}
	{\sc P.~J.~Mumby, A.~Hastings, and H.~J.~Edwards}, {\em Thresholds and the resilience of Caribbean coral reefs}, Nature, 450 (2007), pp.~98--101.
	
	\bibitem{Pe01}
	{\sc L.~Perko}, {\em Differential Equations and Dynamical Systems}, Springer, New York, 2001.
	
	\bibitem{TanLanWei24}
	{\sc M.~Tan, G.~Lan, and C.~Wei}, {\em Mathematical insights into the influence of delay and external recruitment on coral--macroalgae system}, J. Franklin Inst., 361 (2024), 107329.
	
	\bibitem{Zha92}
	{\sc Z.~Zhang, T.~Ding, W.~Huang, and Z.~Dong}, {\em Qualitative Theory of Differential Equations}, American Mathematical Society, Providence, RI, 1992.
	
	\bibitem{ZhYu22}
	{\sc S.~N.~Zhao and S.~L.~Yuan}, {\em A coral reef benthic system with grazing intensity and immigrated macroalgae in deterministic and stochastic environments}, Math. Biosci. Eng., 19 (2022), pp.~3449--3471.
	
\end{thebibliography}
\end{document}